\documentclass[12pt,a4paper]{amsart}

\makeatletter
\@namedef{subjclassname@2020}{\textup{2020} Mathematics Subject Classification}
\makeatother

\usepackage{graphicx} 
\usepackage{amsfonts}
\usepackage{amsmath}
\usepackage{amssymb}
\usepackage{amsthm}
\usepackage{mathrsfs}
\usepackage{xcolor}
\usepackage{mdframed}
\usepackage{lmodern}
\usepackage{lscape}
\PassOptionsToPackage{hyphens}{url}\usepackage{hyperref}
\usepackage[square,numbers]{natbib}
\usepackage{tikz}
\usepackage{caption}
\usepackage{comment}
\usepackage{xurl}

\newcommand{\lean}[1]{{\text{#1}}}
\newcommand{\bdg}{d v^\Sigma_{\bar{\boldsymbol{g}}}}

\newtheorem{thm}{Theorem}[section]
\newtheorem{prop}[thm]{Proposition}
\newtheorem{corr}[thm]{Corollary}
\newtheorem{lem}[thm]{Lemma}

\theoremstyle{definition}

\def\sideremark#1{\ifvmode\leavevmode\fi\vadjust{\vbox to0pt{\vss
 \hbox to 0pt{\hskip\hsize\hskip1em
 \vbox{\hsize3cm\tiny\raggedright\pretolerance10000
  \noindent #1\hfill}\hss}\vbox to8pt{\vfil}\vss}}}%
                        
\theoremstyle{remark}
\newtheorem{rem}[thm]{Remark}

\newcommand{\cc}{\boldsymbol{c}}
\newcommand{\ce}{\mathcal{E}}
\newcommand{\cV}{\mathcal{V}}
\newcommand{\si}{\sigma}
\newcommand{\cT}{\mathcal{T}}

\newcommand{\Up}{\Upsilon}
\newcommand{\bg}{\boldsymbol{g}}

\newcounter{mnotecount}
\newcommand{\mnotex}[1]
{\protect{\stepcounter{mnotecount}}$^{\mbox{\footnotesize $\bullet$\themnotecount}}$
\marginpar{
\raggedright\tiny\em
$\!\!\!\!\!\!\,\bullet$\themnotecount: #1} }

\newcommand{\lpl}{
  \mbox{$
  \begin{picture}(12.7,8)(-.5,-1)
  \put(2,0.2){$+$}
  \put(6.2,2.8){\oval(8,8)[l]}
  \end{picture}$}}

\newcommand{\Addresses}{{
  \bigskip
  \footnotesize

  B.~F.~A.,\,  \textsc{Department of Mathematics, The University of Auckland,
    Private Bag 92019, Auckland 1142, New Zealand}\par\nopagebreak
  \textit{E-mail address},   B.~F.~A, \, : \texttt{ball111@aucklanduni.ac.nz}

  \medskip

    A.~R.~G.,\, \textsc{Department of Mathematics, The University of Auckland,
    Private Bag 92019, Auckland 1142, New Zealand}\par\nopagebreak
  \textit{E-mail address},   A.~R.~G. \, : \texttt{r.gover@auckland.ac.nz}

}}

\title[Higher Willmore Energies]{Higher Willmore Energies from tractor coupled GJMS operators}
\author{Ben F.\ Allen and Rod Gover}
\date{}

\begin{document}

\thanks{B.A. \& A.R.G.\ gratefully acknowledges support from the Royal
  Society of New Zealand via Marsden Grant 19-UOA-008. A.R.G.'s
  contribution to this work was also partially supported by a grant from
  the Simons Foundation and EPSRC Grant Number EP/RO14604/1 during a
  visit to the Isaac Newton Institute for Mathematical Sciences}


\begin{abstract}
We define and construct a conformally invariant energy for closed
smoothly immersed submanifolds of even dimension, but of arbitrary
codimension, in conformally flat Riemannian manifolds. This is a
higher dimensional analogue of the Willmore
energy for immersed surfaces and is given directly via a coupling of the tractor connection to the (submanifold critical) GJMS
operators.
In the case where the submanifold is of
dimension 4 we compare this to other energies, including one found
using a second simple construction that uses $Q$-operators.
\end{abstract}

\maketitle

\section{Introduction}

The Willmore energy of a closed immersed surface $\Sigma$ in Euclidean
$n$-space is given by
\begin{equation}\label{willE}
\int_\Sigma |H|^2 ~ \bdg
\end{equation}
where $H^a$ is the mean curvature vector field. As the integrand is
quadratic in $H$ the Euler-Lagrange equation (from variations of
embedding) has a linear leading term, and this is $\Delta H^a$, up to
a non-zero constant factor, where $\Delta $ is this
normal-bundle-coupled submanifold Laplacian.  A significant part of
the interest in the energy \eqref{willE} stems from the fact that it
is conformally invariant
\cite{Blaschke,WillmoreBook,WillmoreArticle,MaNe}, which means that so
also is the Willmore equation. The functional gradient of
\eqref{willE} is an interesting conformal invariant of surfaces, in
particular because of the linear leading term.

There has been considerable interest in finding analogues of this
energy and equation for submanifolds of higher dimension. For
hypersurfaces (meaning submanifolds of codimension 1) in Euclidean
5-space \cite{Gu05} Guven attacked the problem by setting up an
explict undetermined coefficient problem to search among basic objects
to find a combination that is invariant under conformal motions.  More
recently some constructions of higher Willmore equations and energies
have used {\em holographic} approaches, meaning that each submanifold
is linked to solutions of an appropriate geometric partial
differential on the ambient (or host)  manifold. To the extent
that the solution is uniquely determined by the submanifold, the jets
of the solution capture the data of the submanifold. For hypersurfaces
in conformal manifolds, of any dimension, this was initiated in
\cite{GW-Barcelona} (online as \cite{GW-ann}) based on a singular
Yamabe problem, inspired by \cite{ACF,G-aE}, and then followed up in
detail in a series developing also surrounding theory
\cite{GW,Gr-Vol,GoWa16-1,GoWa16-2,GoWa16-3,GoWa19,GW-LN}; see also \cite{JuhlOrst}.  For
higher codimension embeddings, Graham-Reichert and Zhang have found
higher Willmore energy analogues \cite{GR,YZ} by exploiting what might
be called a 2-step holography that involves a minimal submanifold
problem in Poincar\'e-Einstein manifolds, as studied in the
Graham-Witten work \cite{GW} where already a link to the Willmore
energy was noted.

Each of these constructions mentioned provide higher analogues of the
Willmore equation, as the equations involved are conformally invariant
geometric PDEs that are fully determined by the conformal submanifold
embedding data, and have a linear leading term. For hypersurfaces this
is discussed in detail in \cite{GW-LN}. We will refer to invariant
submanifold action integrals as being of Willmore-type if their
variations with respect to embeddings have such an appropriate leading
term.

The holographic constructions of Willmore energies and related
invariants are interesting because the energies and their
Euler-Lagrange equations are determined indirectly as geometric
invariants of the formal asymtotics of a geometric PDE problem in the
ambient (or host) manifold in which the submanifold is embedded. This
is conceptually powerful as it means these quantities arise as data
in an application that is independently interesting. On the other
hand these holographic constructions do not directly provide a formula
in terms of the underlying conformal embedding. One needs to extract
such formulae from the jets of an asymptotic solution, and that can be
complicated.

This fact provides good motivation to seek simpler direct
constructions of the Willmore-type energies, and that is what we take
up in the current article. Apart from their direct interest for PDE
problems (such as understanding their extrema) the resulting formulae
can inform the holographic programme.

In the 2013 work \cite{V}, Vyatkin used the tractor calculus and a result from
\cite{BrGo05} to construct a curvature quantity $Q$, for four
dimensional hypersurfaces embedded in conformally flat spaces, that is somewhat of a hypersurface analogue of Branson's $Q$-curvature;
see
Lemma $5.2.7$ of \cite{V}. For closed hypersurfaces $\Sigma$, the
associated conformally invariant integral is
\begin{equation*}
  \int_\Sigma \left(\frac{4}{9}D^kI\!I^{\circ~}_{~jk}D_lI\!I^{\circ jl}-4p_j^{~l}I\!I^{\circ}_{~lk}I\!I^{\circ jk}+2\jmath I\!I^{\circ}_{~jk}I\!I^{\circ jk}\right)\bdg,
\end{equation*}
where $I\!I^{\circ}{}_{jk}:=I\!I^{\circ}{}_{jk}{}^aN_a$ is the tracefree
second fundamental form for hypersurfaces, $p_{jk}$ the intrinsic
Schouten tensor of the submanifold, and $\jmath$ is its metric trace.
Here $\bdg$ is the volume form density for the submanifold that is determined by the conformal structure. Vyatkin
shows that (up to a non-zero constant factor) the variation of this
action has leading term $\Delta^2H$, where $H$ is the mean curvature
of the submanifold. So it provides a higher Willmore energy that is
explicit, direct, and conformally invariant by construction.

A main focus of this paper is the construction of  a generalising analogue of this
{Vyatkin energy} to {arbitrary codimension}, by using the tractor
calculus for submanifolds developed in \cite{CuGoS} and the
$Q$-operators of \cite{BrGo03,BrGo05}, and this is done in Section
\ref{QOperatorEnergySection}. A second focus is the construction of
energies of Willmore-type for closed submanifolds of any even
dimension embedded in conformally flat spaces of arbitrary codimension
by using a coupling of the conformal Laplacian operators of \cite{GJMS} (the {\em GJMS operators}), and this can be found in
Section \ref{GJMS Energy section}. A comparison between the GJMS and
$Q$-operator energies is given in Section \ref{Comparing energies in
  dimension four}, and a proof that these energies are of
Willmore-type is given in Section \ref{GJMS is Wt}. An introduction to
the tractor calculus of conformal submanifolds is given in Section
\ref{Intro}.

The following two theorems summarise the results of this paper.
\begin{thm}\label{GJMSEnergyThm}
    Let $\Sigma^m\rightarrow M$ be a closed submanifold of even
    dimension $m$ immersed in a conformally flat Riemannian manifold
    $M$ of arbitrary codimension. There is a conformally invariant
    energy $\Tilde{\mathcal{E}}$ on $\Sigma$ of Willmore-type  defined by
    \begin{equation}\label{GeneralWillmoreenergyy}
        \Tilde{\mathcal{E}}:=\int_{\Sigma}N^A_BP^\nabla_mN^B_A\bdg \, ,
    \end{equation}
    where $N^A_B$ is the normal tractor projector and $P^\nabla_m$ is the intrinsic critical GJMS operator coupled to the ambient tractor connection.
\end{thm}
\noindent The idea behind the construction of the above energy is as follows. On
a conformal manifold of even dimension $m$ there exist the GJMS
operators \cite{GJMS}, and more specifically the critical GJMS
operator, which has leading term $\Delta^{m/2}$. A conformal
submanifold embedded in a conformally flat manifold $M$ is equipped
with a conformally invariant flat connection on the ambient tractor
bundle -- this is simply the pullback to $\Sigma$ of the flat tractor
bundle on $M$. We couple the submanifold critical GJMS operator to
this flat connection, and this gives us an invariant operator defined
on the ambient tractor bundle. Composing this operator with the normal
tractor projector then produces a conformal density of the correct
weight so that the invariant integral above can be made. We refer to
this energy as the GJMS energy. For hypersurfaces a similar idea has been
used in \cite{GoWa16-2,BlGoWa} -- in that setting an extrinsically-coupled
variant of the GJMS operator was used. More recently, Martino
\cite{Mar23} constructed a Willmore-type energy for four-dimensional
closed hypersurfaces in conformally-flat backgrounds using the Paneitz
operator acting on the normal unit tractor.

\begin{thm}\label{QoperatorEnergyThm}
    Let $\Sigma^4\rightarrow M$ be a closed submanifold of dimension four immersed in a conformally flat Riemannian manifold $M$. The energy
    \begin{equation}\label{SecondGWE}
        \mathcal{E}:=\int_{\Sigma}\left(\left(\check{\nabla}_j\mathbb{L}^{jA}{}_B\right)\check{\nabla}^k\mathbb{L}_{kA}{}^B-4p_j{}^k\mathbb{L}^{jA}{}_B\mathbb{L}_{kA}{}^B+2\jmath\mathbb{L}^{jA}{}_B\mathbb{L}_{jA}{}^B\right)\bdg
    \end{equation}
    is a conformally invariant energy of Willmore-type, where $\mathbb{L}_{jA}{}^B$ is the tractor second fundamental form \eqref{TRACSECFF}, $p_{jk}$ is the intrinsic Schouten tensor and $\jmath$ its trace, and $\check{\nabla}$ is the normal-coupled checked tractor connection. 
\end{thm}
\noindent The above energy, which we refer to as the $Q$-operator
energy, is constructed using an interesting application of
$Q$-operators in conjunction with the tractor calculus of
submanifolds. $Q$-operators are operators which can be used together
with closed forms to produce densities on manifolds that, in a
suitable sense, are analogues of the $Q$-curvature. Further details of
this process can be found in Section \ref{FirstQOperator} and in
\cite{BrGo03,BrGo05}. On conformal submanifolds there is a Codazzi
type equation which, when the ambient manifold is conformally flat,
tells us that the tractor analogue of the second fundamental form is
closed with respect to a submanifold tractor connection. By coupling a
specific $Q$-operator on the submanifold to this submanifold tractor
connection, we can use the tractor second fundamental form to
construct a $Q$-like density on the submanifold, and thus an invariant
integral, and this is the quantity \eqref{SecondGWE} above.

The GJMS energy and the $Q$-operator energy are different in dimension
four. These energies are related by
\begin{equation}\label{SuperLastReferencedEquation}
    \Tilde{\mathcal{E}}=-2\mathcal{E}+4\int_{\Sigma}\left(I\!I^{\circ j}{}_b{}^aI\!I^{\circ}{}_j{}^b{}_dI\!I^{\circ}{}_{kc}{}^dI\!I^{\circ kc}{}_a+I\!I^{\circ j}{}_b{}^aI\!I^{\circ}{}_j{}^d{}_aI\!I^{\circ}{}_{kd}{}^cI\!I^{\circ kb}{}_d\right)\bdg,
\end{equation}
where $I\!I^{\circ}{}_{ja}{}^b$ is the tracefree second fundamental
form. The details of the computation for the above equation can be
found in subsection \ref{Comparing energies in dimension four}. Establishing Expression \eqref{SuperLastReferencedEquation} uses that the GJMS energy can  also be viewed as arising from the $Q$-operators, see Propositions \ref{tE-prop} and \ref{tE-QL}.

A family of examples of submanifolds which are critical for these energies are those which are umbilic; see Remark \ref{RemUmb}.

The Willmore-type energy constructed in \cite[{\bf Remark 5.4}]{GR} for four dimensional submanifolds of arbitrary codimension, commonly referred to as the {\em Graham-Reichert energy} $\mathcal{E}_{GR}$, is written below in our notation.
\begin{equation*}
    \begin{split}
        8\mathcal{E}_{GR}:=\int_{\Sigma}\biggl(&(D_jH_b-P_{ja}N^a_b)(D^jH^b-P^{jc}N^b_c)-|\check{P}|^2+\check{J}\\
    &-W^k{}_{akb}H^aH^b-2C^k{}_{ka}H^a-\frac{1}{n-4}B^k{}_k\biggr)\bdg.
    \end{split}
\end{equation*}
Here $C_{abc}:=2\nabla_{[a}P_{b]c}$ is the ambient Cotton tensor and $B_{ab}:=\nabla^aC_{cab}+P^{cd}W_{acbd}$ is the ambient Bach tensor. In Section \ref{GRcompare} we will show that, for four dimensional submanifolds immersed in conformally flat ambient manifolds of arbitrary codimension, the $Q$-operator energy and the Graham-Reichert energy have difference given by
\begin{equation*}
    32\mathcal{E}_{GR}-\mathcal{E}=16\pi^2\chi(\Sigma)+\int_{\Sigma}\left(-4|\mathcal{F}|^2+4\mathrm{f}^2-\frac{1}{2}w_{ijkl}w^{ijkl}\right)\bdg,
\end{equation*}
where $\mathcal{F}_{jk}$ is a conformally invariant tensor called the Fialkow tensor (see Section \ref{CheckedConnetionSection}, or \cite[{\bf Section 3.2.6}]{V}), $\mathrm{f}:=\mathcal{F}_{jk}\bg^{jk}$ is the metric trace of the Fialkow tensor, $\chi\left(\Sigma\right)$ is the Euler characteristic of the submanifold $\Sigma$, and $w_{ijkl}$ is the intrinsic Weyl tensor on $\Sigma$. Notice that this difference above is manifestly conformally invariant.

There are other comparisons between Willmore type energies in dimension four: In \cite[{\bf Section 4}]{BlGoWa} an energy is constructed for hypersurfaces in conformally flat manifolds using the Paneitz operator applied to the normal unit tractor, similar to \cite{GoWa16-2,Mar23}. There they compare this energy to various known Willmore type energies, such as those of Guven and Graham-Reichert. More comparisons are given in \cite{GR}.

Since the main work of this paper there have been posted two further works
with links to the topics here
\cite{Blitz-Silhan,McKoewnetal}.

\section{Notation and preliminaries}\label{Intro}

For simplicity of exposition we will work on manifolds equipped with
Riemannian signature metrics, although with minor adjustments the
theory developed applies in any signature.

\subsection{Some conventions for Riemannian geometry and submanifolds}
\label{conventions}

For tensorial calculations we will use Penrose's abstract index
notation, unless otherwise indicated. We write $\mathcal{E}_a$ and
$\mathcal{E}^a$ as (alternative) notations for, respectively, the
cotangent bundle and the tangent bundle. A contraction of a 1-form
$\omega$ with a tangent vector $v$ is written with a repeated abstract
index $\omega_a v^a$.  Tensor bundles are denoted then by attaching to
the symbol $\mathcal{E}$ indices in a way that encodes the tensor type. For example $\ce_{ab}$ means
$T^*M \otimes T^*M$, while $\mathcal{E}_{(ab)}$ is the abstract index
notation for $S^2T^*M$, the subbundle of symmetric tensors in
$T^*M\otimes T^*M$. Another example is the bundle $\bigwedge^2T^*M$ of skewsymmetric tensors, which in abstract index notation is written as $\mathcal{E}_{[ab]}$.

On a Riemannian $n$-manifold $(M,g)$, our convention for the
Riemann tensor $R_{ab}{}^c{}_d$ is such that
\begin{align} \begin{split} \label{riemdef}
\left[ \nabla_a,  \nabla_b \right] v^c &= R_{ab}{}^c{}_d v^d, \\
\end{split}
\end{align}
where $\nabla_a$ is the Levi-Civita connection of a metric $g_{ab}$ and
$v^c$ any tangent vector field. As is well known
$R_{abcd}=g_{ce}R_{ab}{}^e{}_d$ may be decomposed 
\begin{equation} \label{riemdec}
R_{abcd} = W_{abcd} + 2 \left(g_{c[a}P_{b]d} + g_{d[b}P_{a]c}\right),
\end{equation}
where the completely trace-free part $W_{abcd}$ is called the {\em Weyl
  tensor}. It follows that in dimensions $n\geq 3$ we have
\begin{equation}
  P_{ab} =  \frac{1}{n-2} \left( R_{ab} - \frac{R}{2 \left( n-1 \right)} g_{ab}\right) ,
\end{equation}
where $R_{bc}= R_{ab}{}^a{}_c$ is the Ricci tensor, and its metric
trace $R=g^{ab}R_{ab}$ is scalar curvature. 
We will use $J$ to denote
the metric trace of Schouten, i.e. $J := g^{ab} P_{ab}$.

Given a smooth $n$-manifold $M$, a \emph{submanifold} will mean a
smooth immersion  $\iota : \Sigma \to M$ of a smooth $m$-dimensional
manifold $\Sigma$, where $1 \leq m \leq n-1$, and the image has
\emph{codimension} $d:= n-m$.  Normally we will suppress the
 immersion map and identify $\Sigma$ with its image
$\iota(\Sigma) \subset M$. In this context we refer to $M$ as the \emph{ambient}
manifold.

In our (abstract index) notation, the intrinsic tensor bundles for
immersed Riemannian submanifolds (and their equivalents) are written
with indices denoted with middle latin letters $i,j,k,\dots$, to help
distinguish these bundles from the corresponding ambient bundles. For
example, for the intrinsic tangent bundle we write $\mathcal{E}^j$,
and for the induced metric tensor $g_{ij}$. Here intrinsic refers to
the standard Riemannian objects which can be constructed using the
induced Riemannian metric with respect to the immersion. A
submanifold immersion $\Sigma\rightarrow M$  induces a short exact sequence of bundles on
the submanifold, which we write as
\begin{equation}\label{n-seq}
    0\longrightarrow\mathcal{E}^i\overset{\Pi^a_i}{\longrightarrow}\mathcal{E}^a\longrightarrow \mathcal{N}^b\longrightarrow0,
\end{equation}
where $\Pi^a_i$ denotes the pushforward map on the tangent bundle, and $\mathcal{N}^a$ is the normal bundle. Here, and
frequently in similar situations below, we write simply $\ce^a$ rather
than $\ce^a|_\Sigma=TM|_\Sigma$, as the restriction to the submanifold
is clear by context.

In the presence of a Reimannian metric $g$ (or even its conformal
class, as discussed below) the exact sequence \eqref{n-seq} defining
the normal bundle splits and we can identify $\mathcal{N}^b$ with a
subbundle of $\ce^a$, along $\Sigma$. We write $N^b_a:\ce^a|_\Sigma
\to \mathcal{N}^b$ for the orthogonal projection. This enables us to define the {\em normal connection} on $\mathcal{N}^b$
by
$$
\nabla_i^{\perp}v^a := N^a_b\nabla_i v^b ,
$$ where $\nabla_i$ is the pullback to $\Sigma$ of the ambient
Levi-Civita connection.

Our convention for the second fundamental form
$I\!I_{ij}{}^a$ is so that the Gauss formula is
\begin{equation*}
    \nabla_iu^a=\Pi^a_jD_iu^j+I\!I_{ij}{}^au^j,
\end{equation*}
where $u^j\in\Gamma\left(\mathcal{E}^j\right)$ is an intrinsic vector field, $u^a:=\Pi^a_ju^j$, and $D_i$ is the Levi-Civita connection of the metric $g_{ij}$. The corresponding Weingarten formula is
\begin{equation*}
    \nabla_iv^a=\nabla_i^{\perp}v^a-\Pi^a_jI\!I_i{}^j{}_bv^b,
\end{equation*}
where $v^a\in\Gamma\left(\mathcal{N}^a\right)$ is a normal vector
field and $\nabla^{\perp}$ is the induced connection on the normal
bundle called the normal connection. It will be convenient to extend
the intrinsic Levi-Civita connection to act on sections of the ambient
tangent bundle by orthogonal decomposition and coupling to the normal
connection. That is the extension is defined by
\begin{equation*}
    D_ju^a:=\Pi^a_iD_j\left(\Pi^i_bu^b\right)+\nabla_j^{\perp}\left(N^a_bu^b\right),
\end{equation*}
where $u^a\in\Gamma\left(\mathcal{E}^a\right)$.

Important submanifold quantities include the curvature tensors of the induced Levi-Civita connection $D$, such as the intrinsic Riemann tensor $r_{ijkl}$ defined by
\begin{equation*}
    r_{ij}{}^k{}_lu^l:=[D_i,D_j]u^k
\end{equation*}
for $u^k\in\Gamma\left(\mathcal{E}^j\right)$, the intrinsic Schouten tensor $p_{jk}$ defined by
\begin{equation}\label{int-Sch}
    p_{jk}:=\frac{1}{m-2}\left(r_{jk}-\frac{r_{il}g^{il}}{2(m-1)}g_{jk}\right)
\end{equation}
where $r_{jl}:=r_{kj}{}^k{}_l$ is the intrinsic Ricci tensor, and the metric trace of the intrinsic Schouten tensor $\jmath$. Extrinsic quantities include the mean curvature vector $H^a:=\frac{1}{m}g^{ij}I\!I_{ij}{}^a$, and the tracefree second fundamental form $I\!I^{\circ}{}_{ij}{}^a$, or umbilicity tensor, defined by
\begin{equation*}
    I\!I^{\circ}{}_{ij}{}^a:=I\!I_{ij}{}^a-g_{ij}H^a.
\end{equation*}
The umbilicity tensor is a conformal invariant of the submanifold,
along with the intrinsic Weyl tensor $w_{ijkl}$, which is the
completely trace-free part of $r_{ijkl}$, the ambient Weyl tensor
$W_{abcd}$, and the normal projection operator
$N^a_b:\mathcal{E}^b\rightarrow\mathcal{N}^a$. For further details see
\cite[{\bf Section 3.1}]{CuGoS}. Note that \eqref{int-Sch} is
evidently not defined for surfaces and curves. In fact there are
replacements, see \cite[{\bf Section 3.5}]{CuGoS}, but in the current
work we shall be mainly interested in the case of $m\geq 4$.

\subsection{The tractor bundle and connection}\label{t-sec}

For most of our discussion it will be convenient to work in the
setting of conformal manifolds. By a conformal manifold $(M,\cc)$ we
mean a smooth manifold equipped with an equivalence class $\cc$ of
Riemanian metrics, where $g_{ab}$, $\widehat{g}_{ab} \in \cc$ means
that $\widehat{g}_{ab}=\Omega^2 g_{ab}$ for some smooth positive
function $\Omega$. On a general conformal manifold $(M,\cc)$, there is
no distinguished connection on $TM$. But if $n\geq 3$ there is an
invariant and canonical connection on a closely related bundle, namely
the conformal tractor connection on the standard tractor bundle.

Here we review the basic conformal tractor calculus, see
\cite{CuGo13,CuGoS} for more details.
Unless stated otherwise, calculations will be done with the use of  $g \in
\cc$.

Writing  $ \Lambda^{n} TM$ for the top exterior power of the
tangent bundle, note that  its square
$\mathcal{K}:=(\Lambda^{n} TM)^{\otimes 2}$  is
canonically oriented and so we can take compatibly oriented roots of it: Given
$w\in \mathbb{R}$ we denote
\begin{equation} \label{cdensities}
\ce[w]:=\mathcal{K}^{\frac{w}{2n}} ,
\end{equation}
and refer to this as the bundle of conformal densities. For any vector
bundle $\cV$, we then write $\cV[w]$ as a shorthand for
$\cV[w]:=\cV\otimes\ce[w]$.

There is a
canonical section $\bg_{ab}\in \Gamma(\ce_{(ab)}[2])$ with the
property that for each positive section $\si\in \Gamma(\ce_+ [1])$
(called a {\em scale}) $g_{ab}:=\si^{-2}\bg_{ab}$ is a metric in
$\cc$. Moreover, the Levi-Civita connection of $g_{ab}$ preserves
$\si$ and therefore $\bg_{ab}$. Thus we typically use the
conformal metric to raise and lower indices, even when we are chosing
a particular metric $g_{ab} \in \cc$ and its Levi-Civita connection
for calculations.  This simplifies 
computations, and so we  do that without further
mention. 

By examining the Taylor series of sections of $\ce[1]$ we can recover the
jet exact sequence at 2-jets for this bundle, 
\begin{equation}\label{J2}
0\to \ce_{(ab)}[1]\stackrel{\iota}{\to} J^2\ce[1]\to J^1\ce[1]\to 0 .
\end{equation}
Note that the bundle $J^2\ce[1]$ and its sequence \eqref{J2} are
canonical objects on any smooth manifold. However on a manifold with a conformal
structure $\cc$ we have also the orthogonal decomposition of $\ce_{ab}[1]$
into trace-free and trace parts
\begin{equation}
\ce_{ab}[1]= \ce_{(ab)_0}[1]\oplus \bg_{ab}\cdot \ce[-1] .
\end{equation}
This means that we can take a  quotient $J^2\ce[1]$ by the image of
$\ce_{(ab)_0}[1]$ under $\iota$ (in (\ref{J2})). The resulting
quotient bundle is denoted $\cT^*$, or $\ce_A$ in abstract indices,
and called the {\em conformal cotractor bundle}. Since the
jet exact sequence at 1-jets (of $\ce[1]$) is given,
$$ 
0\to \ce_{b}[1]\stackrel{\iota}{\to} J^1\ce[1]\to \ce[1]\to 0,
$$ 
it follows  that $\cT^*$ has a composition series
\begin{equation}\label{filt}
\cT^*=\ce[1]\lpl \ce_a [1] \lpl \ce[-1] ,
\end{equation}
where the notation 
means that $\ce[-1]$ is a subbundle of $\cT^*$ and the quotient
of $\cT^*$ by this (which is $J^1\ce[1]$) has $\ce_a [1]$ as a
subbundle, whereas there is a canonical projection $X:\cT^*\to \ce[1]$.
In abstact indices we write $X^A$ for this map and call it the {\em canonical tractor}.

Given a choice of metric $g \in\cc$, the formula
\begin{equation}\label{thomas-D}
  \sigma \mapsto
  \begin{pmatrix}
    \sigma \\
    \nabla_a \sigma \\
    - \frac{1}{n} \left( \Delta + J \right) \sigma
  \end{pmatrix}
\end{equation}
(where $\Delta $ is the Laplacian $\nabla^a\nabla_a$) gives a second-order differential operator on $\ce[1]$ which is a linear map $J^2 \ce[1] \to \ce[1] \oplus \ce_a [1] \oplus \ce[-1]$ that clearly factors through $\cT^*$ and so, in this way, $g$ determines an isomorphism
\begin{equation}\label{T_isom}
  \cT^* \stackrel{\sim}{\longrightarrow} {[\cT^*]}_g = \ce[1] \oplus \ce_a [1] \oplus \ce[-1].
\end{equation}

We will use~\eqref{T_isom} to split the
tractor bundles without further comment.  Thus, given $g \in \cc$, an
element $V_A$ of $\ce_A$ may be represented by a triple
$(\si,\mu_a,\rho)$, or equivalently by
\begin{equation}\label{Vsplit}
  V_A=\si Y_A+\mu_a Z_A{}^a+\rho X_A.
\end{equation}
The last display defines the algebraic splitting operators
$Y:\ce[1]\to \cT^*$ and $Z :T^*M[1]\to \cT^*$ (determined by the
choice $g_{ab} \in \cc$) which may be viewed as sections $Y_A\in
\Gamma(\ce_A[-1])$ and $Z_A{}^a\in \Gamma(\ce_A{}^a [-1])$.  We call
these sections $X_A, Y_A$ and $Z_A{}^a$ \emph{tractor projectors}.

By construction the tractor bundle is conformally invariant, i.e. it
is determined by $(M,\cc)$ and indpendent of any choice of
$g\in\cc$. However the splitting \eqref{Vsplit} is not. Considering
the transformation of the operator \eqref{thomas-D}, determining the
splitting, we see that if $\widehat{g}=\Omega^2 f$ the components of an
invariant section of $\cT^*$ should transform according to:
\begin{equation}\label{ttrans}
[\cT^*]_{\widehat{g}}\ni \left(\begin{array}{c}
\widehat{\sigma}\\
\widehat{\mu}_b\\
\widehat{\rho}
\end{array}\right)=
\left(\begin{array}{ccc}
1 & 0 & 0\\
\Upsilon_b & \delta^c_b & 0\\
-\frac{1}{2}\Upsilon^2 & -\Upsilon^c & 1
\end{array}\right)\left(\begin{array}{c}
\sigma\\
\mu_c\\
\rho
\end{array}\right)~\sim~ \left(\begin{array}{c}
\sigma\\
\mu_b\\
\rho
\end{array}\right)\in [\cT^*]_g,
\end{equation}
where $\Up_a=\Omega^{-1}\nabla_a \Omega$. This transformation of
triples is the characterising property of an invariant tractor
section. Equivalent to the last display is the rule for how the
algebraic splitting operators transform
\begin{equation}\label{XYZtrans}
\textstyle
\widehat{X}_A=X_A, \quad  \widehat{Z}_A{}^{b}=Z_A{}^{b}+\Up^bX_A, \quad
\widehat{Y}_A=Y_A-\Up_bZ_A{}^{b}-\frac12\Up_b\Up^bX_A \, .
\end{equation}

Given a metric $g \in \cc$, and the corresponding splittings, as above, the tractor connection is given by the formula
 \begin{equation}\label{tr-conn}
 \nabla_a^\cT  \begin{pmatrix}
    \sigma \\
    \mu_b \\
    \rho
  \end{pmatrix} :=\begin{pmatrix}
   \nabla_a \sigma-\mu_a \\
    \nabla_a\mu_b+P_{ab}\si+\bg_{ab}\rho \\
    \nabla_a\rho- P_{ac}\mu^c
  \end{pmatrix} ,
\end{equation}
where on the right hand side the $\nabla$s are the Levi-Civita connection of $g$. Using the transformation of components, as in \eqref{ttrans}, and also the conformal transformation of the Schouten tensor,
\begin{equation}\label{Ptrans}
P^{\widehat{g}}_{\phantom{\widehat{g}}ab}=P_{ab}-\nabla_a\Upsilon_b
+\Upsilon_a\Upsilon_b-\frac{1}{2}g_{ab}\Upsilon_c\Upsilon^c,  \quad \widehat{g}=\Omega^2 g ,
\end{equation}
reveals that the triple on the right hand side transforms as a 1-form
taking values in $\cT^*$ -- i.e. again by \eqref{ttrans} except
twisted by $\ce_a$. Thus the right hand side of \eqref{tr-conn} is the
splitting into slots of a conformally invariant connection
$\nabla^\cT$ on a section of the bundle $\cT^*$.

The tractor bundle is also equipped with a conformally invariant
signature $(n+1,1)$ metric $h_{AB} \in \Gamma \left(\ce_{(AB)}
\right)$ (where, note, $\ce_{(AB)}$ is the abstract index notation for $S^2\cT^*$), defined as quadratic form by the mapping
\begin{equation}
[V_A]_g = \begin{pmatrix}
    \sigma \\
    \mu_a \\
    \rho
  \end{pmatrix} \mapsto \mu_a \mu^a +2  \si \rho =: h \left(V,V \right). 
\end{equation}
It is easily checked that this {\em tractor metric} $h$ is conformally
invariant and 
is preserved by $\nabla_a^\cT$, i.e. $ \nabla_a^\cT h_{AB} =0$. Thus
it makes sense to use $h_{AB}$ (and its inverse) to raise and lower
tractor indices, and we do this henceforth. In
particular $X^A=h^{AB}X_B$ is the canonical tractor (and hence our use
of the same symbol). For computations the table of Figure
\ref{tmf} is useful,
\begin{figure}[ht]
$$
\begin{array}{l|ccc}
& Y^A & Z^{Ac} & X^{A}
\\
\hline
Y_{A} & 0 & 0 & 1
\\
Z_{Ab} & 0 & \delta_{b}{}^{c} & 0
\\
X_{A} & 1 & 0 & 0
\end{array}
$$
\caption{Tractor inner product}
\label{tmf}
\end{figure}
and we see that $h$ may be decomposed into a sum of projections
$$
h_{AB}=Z_A{}^aZ_{B}{}^b\bg_{ab}+X_AY_B+Y_AX_B\, .
$$

For computations it is also useful to note that the 
tractor connection is determined by its action on the splitting operators:
\begin{align}
    \nabla_aX_B=&\bg_{ab}Z_B^b\label{TracConn1}\\
    \nabla_aZ_B^b=&-\delta_a^bY_B-P_a{}^bX_B\label{TracConn2}\\
    \nabla_aY_B=&P_{ab}Z_B^b.\label{TracConn3}
\end{align}

We refer to any of the following bundles
\begin{equation*}
    \mathcal{E}^{A_1\cdots A_r}_{B_1\cdots B_s}:=\bigotimes_{i=1}^r\mathcal{E}^{A_i}\otimes\bigotimes_{j=1}^s\mathcal{E}_{B_j}
\end{equation*}
for non-negative integers $r$ and $s$, as tractor bundles, and sections of these bundles as tractors (of rank $(r,s)$). Subbundles, such as the symmetric, skewsymmetric, and tracefree tractor bundles, are defined in the obvious way.

\subsection{Tractor Curvature}

The curvature of the tractor
connection $\Omega_{ab}{}^C{}_D$  is defined by
\begin{equation}
    \Omega_{ab}{}^C{}_DV^D=2\nabla_{[a}\nabla_{b]}V^C,
\end{equation}
where $V^C\in\Gamma\left(\mathcal{E}^C\right)$. It follows from the
conformal invariance of the tractor connection that the tractor
curvature is conformally invariant.

We can compute the tractor curvature in terms of the splitting
operators. Using the definitions \ref{TracConn1}, \ref{TracConn2}, and
\ref{TracConn3} of the tractor connection on the splitting operators
we get
\begin{align*}
    \Omega_{ab}{}^C{}_DY^D=&Z^C_cY_{ab}{}^c\\
    \Omega_{ab}{}^C{}_DZ^D_d=&Z^C_cW_{ab}{}^c{}_d-X^CY_{abd}\\
    \Omega_{ab}{}^C{}_DX^D=&0,
\end{align*}
where $Y_{abc}=2\nabla_{[a}P_{b]c}$ is the Cotton tensor. Then
\begin{align}\label{AmbientCurvature2}
    \Omega_{ab}{}^C{}_D=&Z^C_cX_DY_{ab}^{~~c}+Z^C_cZ_D^dW_{ab~d}^{~~c}-X^CZ_D^dY_{abd}
\end{align}
follows from table \ref{tmf}. Recall that a conformal manifold is
conformally flat whenever the Weyl tensor $W_{abcd}$ and the Cotton
tensor $Y_{abc}$ vanish. We see that the tractor connection is flat if
and only if the conformal manifold is conformally flat.

\subsection{Tractor calculus on conformal submanifolds}

Here we present the key introductory tractor calculus for a conformal
submanifold $\Sigma^m\rightarrow (M^n,\cc)$ of dimension $m\geq3$, following \cite{CuGoS}. The
tractor calculus of conformal submanifolds of dimension $m=1,2$ is
discussed in detail in \cite[Section 3.5]{CuGoS}.

We use middle Latin capital letters $I,J,\dots,K$ for indices of
intrinsic tractor bundles on $\Sigma$. For example, the intrinsic
tractor bundle and metric are $\mathcal{E}^J$ and $h_{JK}$. We write
$D$ for the intrinsic tractor connection. Ambient tractor bundles will
refer to the tractor bundles of $M$ (usually, by context, restricted
to $\Sigma$), and will be adorned with indices from the early part of
the alphabet, $A,B,C,\cdots$.

\subsection{Normal and tangent tractor bundle}

The \lean{normal tractor bundle} $\mathcal{N}^A$ is a subbundle of the
ambient tractor bundle $\mathcal{E}^A$. (More precisely, it is a section of $\mathcal{E}^A|_\Sigma$. But as mentioned we shall omit the explicit denoting of the restriction when it is clear by context.) It is defined for each
$g\in\cc$ as the image of the map
\begin{align*}
    N^A_a:\mathcal{N}^a[-1]&\longrightarrow\mathcal{E}^A\\
    v^a&\mapsto\left(\begin{array}{c}
        0\\
        v^a\\
        H_bv^b
    \end{array}\right)
\end{align*}
This map is independent of the choice $g$. Using the ambient tractor
metric we define the tangent tractor bundle
$\overline{\mathcal{E}}{}^A$ as the orthogonal complement of
$\mathcal{N}^A$ in $\mathcal{E}^A$. This gives the direct sum
\begin{equation*}
    \mathcal{E}^A=\overline{\mathcal{E}}{}^A\oplus\mathcal{N}^A.
\end{equation*}
The projection operators onto each subbundle are written $\Pi^A_B$ and $N^A_B$, and are called the \lean{tangent tractor projector} and \lean{normal tractor projector} respectively.

As each metric $g_{ab}$ in $\cc$ induces a metric $g_{ij}$ on $\Sigma$
it follows that $\cc$ induces a conformal structure
$\overline{\cc}$ on $\Sigma$, and hence
$(\Sigma,\overline{\cc})$ has an intrinsic conformal tractor bundle $\ce^J$ and
connection $D_i$.  There is a conformally invariant isomorphism between
$\overline{\mathcal{E}}{}^A$ and the intrinsic tractor bundle
$\mathcal{E}^J$. This is given in (\cite{CuGoS}, Theorem $3.5$). We
shall write this isomorphism as
$\Pi^A_J:\mathcal{E}^J\rightarrow\overline{\mathcal{E}}{}^A$, and
write $\Pi^J_A:\overline{\mathcal{E}}{}^A\rightarrow\mathcal{E}^J$ for
the inverse.  Note that $\Pi^A_B=\Pi^A_J\Pi^J_B$ and
$N^A_B=N^A_aN^a_B$ where $N^a_A:=N^B_b\bg^{ab}h_{AB}$.

\subsection{Intrinsic and extrinsic tractor connections}\label{CheckedConnetionSection}

Let $Y^J$, $Z^J_j$, and $X^J$ be the intrinsic tractor projectors in a scale $g\in\overline{\cc}$. The intrinsic tractor connection $D$ on $\Sigma$ is given in terms of these projectors by
\begin{align*}
    D_iY^J=&p_i{}^jZ^J_j\\
    D_iZ^J_j=&-\bg_{ij}Y^J-p_{ij}X^J\\
    D_iX^J=&Z^J_i,
\end{align*}
where $p_{ij}$ is the intrinsic Schouten tensor, and we write $D_i$
also for the intrinisic Levi-Civita as well as its coupling to the
intrinsic tractor connection. The {\em checked tractor connection}
$\check{\nabla}$ is a connection on $\mathcal{E}^J$ defined using the
isomorphisms $\Pi^A_J$ and $\Pi^J_A$, and the ambient tractor
connection. This is given by
\begin{equation}\label{CheckedTRACConn}
    \check{\nabla}_iU^J:=\Pi^J_A\nabla_iU^A
\end{equation}
where $U^J\in\Gamma\left(\mathcal{E}^J\right)$ and $U^A:=\Pi^A_JU^J$. It is easily verified that this prescription determines a connection on $\mathcal{E}^J$. The action of $\check{\nabla}$ on the intrinsic tractor projectors is given by
\begin{equation}\label{CheckedProjectors}
\begin{split}
    \check{\nabla}_iY^J=&\check{P}_i{}^jZ^J_j\\
    \check{\nabla}_iZ^J_j=&-\bg_{ij}Y^J-\check{P}_{ij}X^J\\
    \check{\nabla}_iX^J=&Z^J_i,
\end{split}
\end{equation}
where
$\check{P}_{ij}:=\Pi^a_i\Pi^b_jP_{ab}+H_cI\!I^{\circ}{}_{ij}{}^c+\frac{1}{2}H^2\bg_{ij}$
is the {\em Schouten-Fialkow tensor}.
The formulae above
for $\check{\nabla}$ on the intrinsic tractor projectors can be
calculated directly from the definition of the checked tractor
connection, or they can be recovered using the action of the intrinsic
tractor connection on the intrinsic tractor projectors and the
tangential contorsion introduced in the following subsection.

\subsubsection{Tangent tractor contorsion and the Fialkow tensor}

The two tractor connections $D$ and $\check{\nabla}$ on conformal submanifolds of dimension $m\geq3$ are in general not the same. The difference between these connections will produce a conformally invariant (tractor) contorsion, which we call the {\em tangent tractor contorsion}. This is the tractor $\mathbb{S}_j{}^K{}_L$ defined by
\begin{equation*}
    \mathbb{S}_j{}^K{}_LU^L:=D_jU^K-\check{\nabla}_jU^K
\end{equation*}
for $U^J\in\Gamma\left(\mathcal{E}^J\right)$. We know how the connections $D$ and $\check{\nabla}$ act on the submanfold splitting operators, so we can get an explicit formula for $\mathbb{S}_j{}^K{}_L$. This is
\begin{equation*}
    \mathbb{S}_j{}^K{}_L=X^KZ_L^l\mathcal{F}_{jl}-Z^K_kX_L\mathcal{F}_j{}^k.
\end{equation*}
Here $\mathcal{F}_{jk}:=\check{P}_{jk}-p_{jk}$ is a conformally invariant tensor called the {\em Fialkow tensor}. It is given by Equation $3.36$ in \cite{CuGoS}, where they show that
\begin{equation*}
    \mathcal{F}_{jk}=\Pi^a_j\Pi^b_kP_{ab}+H_cI\!I^{\circ}{}_{jk}{}^c+\frac{1}{2}H^2\bg_{jk}-p_{jk}.
\end{equation*}

\subsection{Tractor second fundamental form} 

The tractor second fundamental form $\mathbb{L}_{jK}{}^A$ is the
contorsion between the checked tractor connection and the ambient
tractor connection on sections of $\overline{\mathcal{E}}{}^A$. It is
defined by the equation below, called the {\em tractor Gauss formula}.
\begin{equation*}
    \nabla_jU^A=\Pi^A_K\check{\nabla}_jU^K+\mathbb{L}_{jK}{}^AU^K.
\end{equation*}
Here $U^K\in\Gamma (\mathcal{E}^K )$ and $U^A:=\Pi^A_KU^K$. The tractor second fundamental form can be given in terms of the tangent tractor projector by
\begin{equation}\label{TRACSECFF}
    \mathbb{L}_{jA}{}^B=\Pi^C_A\nabla_j\Pi^B_C.
\end{equation}
Often we will write the equivalent $\mathbb{L}_{jK}{}^B=\Pi^A_K\mathbb{L}_{jA}{}^B$ for the tractor second fundamental form. With respect to a splitting $g\in\overline{\cc}$, (applied to the lower tractor index) $\mathbb{L}_{jK}{}^A$ can be written as
\begin{equation}\label{ExplicitSFF}
    \mathbb{L}_{jK}{}^A=\left(\begin{array}{c}
        0\\
        I\!I^{\circ}{}_{jk}{}^a\\
        -D_jH^a+\Pi^b_jN^a_cP_b{}^c
    \end{array}\right)N^A_a,
\end{equation}
which is proved in \cite[Theorem 3.14]{CuGoS}. The {\em tractor Weingarten formula} is the equation
\begin{equation*}
    \nabla_jV^A=\nabla^{\mathcal{N}}_jV^A-\mathbb{L}_j{}^A{}_BV^B
\end{equation*}
where $V^A\in\Gamma\left(\mathcal{N}^A\right)$ and
$\nabla_j^{\mathcal{N}}V^A:=N^A_B\nabla_jV^B$ is the ambient tractor
connection projected onto the normal tractor bundle, which shall call
the normal-projected tractor connection.
We extend the checked tractor
connection to act on sections of the ambient tractor bundle by
coupling to $\nabla^{\mathcal{N}}$. For such a section $U^A$, this
extension is defined by
\begin{equation}\label{55}
    \check{\nabla}_jU^A:=\Pi^A_J\check{\nabla}_j\left(\Pi^J_BU^B\right)+\nabla_j^{\mathcal{N}}\left(N^A_BU^B\right),
\end{equation}
and an important property of this connection is $\check{\nabla}_j\left(\Pi^A_JU^J\right)=\Pi^A_J\check{\nabla}_jU^J$. Using the tractor Gauss and Weingarten formulae it is not hard to show that
\begin{equation}\label{1tractorGaussEq1}
    \nabla_jU^A=\check{\nabla}_jU^A+\mathbb{L}_{jB}{}^AU^B-\mathbb{L}_j{}^A{}_BU^B
\end{equation}
where $U^A\in\Gamma\left(\mathcal{E}^A\right)$. We call $\check{\nabla}$ in Equation \eqref{55} above the {normal-coupled checked tractor connection}.

The following lemma will be useful for our computations in Section \ref{HigherWillmores}.

\begin{lem}\label{LastNeededReferencedLemma}
    The tractor $N^A_a$ is parallel with respect to the normal-coupled checked tractor connection $\check{\nabla}$, coupled to the conformally invariant connection on the weight one conormal bundle.
\end{lem}
\begin{proof}
We will show that $\check{\nabla}_j\left(N^A_bv^b\right)=N^A_bD_jv^b$,
where $D$ is the normal connection on the (weight $-1$) normal bundle, and
$v^b$ is a section of this bundle. This is sufficient to show that
$N^A_a$ is parallel.
Given an {\em embedded} sbmanifold $\Sigma$ we can always find a metric $g\in \cc$ so that $\Sigma$ is minimal, meaning $H^a=0$, see   \cite{G-aE,CuThesis} or \cite[Proposition 3.1]{CuGoS}. In such a scale $N^A_bv^b$ can
be written as
\begin{equation*}
    N^A_bv^b=\left(\begin{array}{c}
        0\\
        v^a\\
        0
    \end{array}\right).
\end{equation*}
It follows from the definition of the tractor connection, and by the Weingarten formula, that
\begin{equation*}
    \nabla_j\left(N^A_bv^b\right)=\left(\begin{array}{c}
        0\\
        D_jv^a-I\!I_j{}^a{}_bv^b\\
        -P_{cb}\Pi^c_jv^b
    \end{array}\right)=N^A_bD_jv^b-T^A_j,
\end{equation*}
where $T^A_j\in\Gamma\left(\overline{\mathcal{E}}{}^A_i\right)$. Since $\check{\nabla}$ is coupled to the normal
connections on both $\mathcal{N}^a$ and $\mathcal{N}^A$, we must have
$\check{\nabla}_j\left(N^A_bv^b\right)=N^A_B\nabla_j\left(N^B_bv^b\right)$,
and therefore
\begin{equation*}
    \check{\nabla}_j\left(N^A_bv^b\right)=N^A_bD_jv^b.
\end{equation*}
\end{proof}

\subsection{Conformal submanifold structure equations}

The structure equations for conformal submanifolds relate the ambient
tractor curvature to the intrinsic tractor curvature. For these
equations see \cite{CuGoS}, respectively Equations $3.42$, $3.43$, and
$3.44$ there.  For our computations in Section \ref{HigherWillmores}
it is convenient to write these structure equations in terms of the
checked tractor connection, rather than the intrinsic tractor
connection, which is a simple (and simplifying) modification involving
adding terms of the tangent tractor contorsion.

Let $\check{\Omega}_{ijKL}$ be the tractor curvature defined by the
checked tractor connection $\check{\nabla}$, and
$\Omega^{\mathcal{N}}_{ijAB}$ the tractor curvature of the normal
tractor connection. The structure equations for conformal submanifolds
are
\begin{align}
    &\Omega_{ijKL}=\check{\Omega}_{ijKL}-2\mathbb{L}_{[i|K}{}^C\mathbb{L}_{|j]LC}\label{CGE}\\
    &\Omega_{ijAK}N^A_B=2\check{\nabla}_{[i}\mathbb{L}_{j]KB}\label{ConformalCodazzi-Mainardi}\\
    &\Omega_{ijAB}N^A_CN^B_D=\Omega^{\mathcal{N}}_{ijCD}-2\mathbb{L}_{[i|}{}^L{}_C\mathbb{L}_{|j]LD}.\nonumber
\end{align}
We call these respectively the {tractor Gauss Equation}, the {tractor
  Codazzi-Mainardi Equation}, and the {tractor Ricci Equation}.

\section{Higher Willmore energies}\label{HigherWillmores}

Here we construct the $Q$-operator energy, which is defined for
submanifolds of dimension four embedded in conformally flat manifolds
of arbitrary codimension. We also construct here the GJMS-coupled
energies for submanifolds of arbitrary even dimension embedded in
conformally flat manifolds of arbitrary codimension.

In fact these two constructions are intimately linked via the natural
linear differential $Q$-operators of \cite{BrGo03,BrGo05}. We thus
begin the next section \ref{QOperatorEnergySection} by introducing
some key aspects of the $Q$-operators. This link enables a comparison
of the $Q$-operator energy and the GJMS energy in dimension four, as
well as a proof that all these energies are of Willmore-type and these
directions are taken up subsequently.

\subsection[$Q$-Operator Energy]{$\boldsymbol{Q}$-Operator Energy}\label{QOperatorEnergySection} 
Here we use a $Q$-operator to construct a global conformal invariant
for four-dimensional closed submanifolds immersed in a conformally
flat manifold of arbitrary codimension. We will give an explicit
formula for the invariant, and show that it is of Willmore-type.

The $Q$-operators of \cite{BrGo03} are differential operators
(available in even dimensions) that, when applied to closed
differential forms, have a conformal transformation of the same form
as the Branson $Q$-curvature. They strictly generalise the latter and
to give context we here recall some other general facts.  However for
our later constructions we shall only need the $Q$-operator on 1-forms
in even dimensions. We partly follow \cite{Go04} in the discussion
here.

Let $(M^n,\cc)$ be a closed
 conformal manifold of dimension $n$. Write $\mathcal{E}^k$ for the
 bundle of $k$-forms on $M$ and $\mathcal{E}_k:=\mathcal{E}^k[n-2k]$
 for the bundle of $k$-forms of weight $n-2k$. This notation for these
 bundles of $k$-forms comes from the following duality. The pairing
 \begin{align*}
     \Gamma\left(\mathcal{E}^k\right)\times\Gamma\left(\mathcal{E}_k\right)&\longrightarrow\mathbb{R}\\
     (\alpha,\beta)&\mapsto\int_M\boldsymbol{\langle}\alpha,\beta\boldsymbol{\rangle}\bdg
 \end{align*}
 is conformally invariant, where $\bdg$ is the density-valued volume form of the conformal metric $\bg$ and $\boldsymbol{\langle}\cdot,\cdot\boldsymbol{\rangle}$ is the inner-product of $k$-forms defined by the conformal metric. Using this pairing we can define the {\em codifferential} $\delta:\Gamma\left(\mathcal{E}_{k+1}\right)\rightarrow\Gamma\left(\mathcal{E}_k\right)$ as the formal adjoint of the exterior derivative $d$. For sections $\alpha\in\Gamma\left(\mathcal{E}^k\right)$ and $\beta\in\Gamma\left(\mathcal{E}_{k+1}\right)$, $\delta\beta$ is given by
 \begin{equation*}
     \int_M\boldsymbol{\langle}\alpha,\delta\beta\boldsymbol{\rangle}\bdg=\int_M\boldsymbol{\langle}d\alpha,\beta\boldsymbol{\rangle}\bdg.
 \end{equation*}
 By definition the codifferential is conformally invariant on sections of $\mathcal{E}_k$.

 We now specialise to when the dimension $n$ is even. Fix a metric
 $g\in\cc$. The {\em $k^{th}$ $Q$-operator}
 $Q^g_k:\Gamma\left(\mathcal{E}^k\right)\rightarrow\Gamma\left(\mathcal{E}_k\right)$,
 $0\leq k\leq \frac{n}{2}-1$
 is a linear differential operator, first
 constructed  in \cite{BrGo03} (and see  also \cite{Aubry-Guillarmou}) with some low order formulae
 computed explicitly in \cite{BrGo05}, which has the form
 \begin{equation}\label{FormOfQ}
     Q^g_k=\left(d\delta\right)^{n/2-k}+lots,
 \end{equation}
 where $lots$ means lower order combinations of $d$ and $\delta$. $Q^g_k$ is not conformally invariant, but satisfies an interesting transformation property when acting on closed $k$-forms. For a change of metric $g\mapsto\widehat{g}:=e^{2\Upsilon}g$, where $\Upsilon\in C^{\infty}(M)$, this transformation is
 \begin{equation}\label{Qop-t}
     Q^{\hat{g}}_k u=Q^g_k u+ \beta\cdot \delta Q^g_{k+1} d\left(\Upsilon u \right),
 \end{equation}
 where $ u\in\mathcal{E}^k$ is a closed $k$-form, $Q^g_{k+1}$ is the
 $(k+1)^{th}$ $Q$-operator in the scale of $g$ (with $Q^g_\frac{n}{2}
 :=1$), $\beta$ a non-zero constant, and each operator acts on all to its right.

 Recall that Branson's $Q$-curvature \cite{BrOsted,tom-sharp}
 is a natural
scalar curvature on even dimensional Riemannian manifolds $n$ with a conformal
weight $-n$ and a conformal transformation of the form
 \begin{equation}\label{Q-t}
     Q^{\hat{g}}=Q^g+\delta Nd \Upsilon ,
 \end{equation}
 where
 $N:\Gamma\left(\mathcal{E}^1\right)\rightarrow\Gamma\left(\mathcal{E}_1\right)$
 is some natural (or geometric) linear differential operator depending
 on $g$ and again $\delta Nd$ is to be read as a composition of the
 differential operators $\delta$, $N$, and $d$. The $Q$-operators
 generalise the $Q$ curvature, in that the latter arises from $Q_0$
 acting on the constant function 1, and because of corresponding
 transformation formula \eqref{Qop-t}.

We shall say that any
 density $Q^g$ of weight $-n$ is {\em $Q$-like} if it has a conformal
 transformation of the form \eqref{Q-t}.  The
 $k^{th}$ $Q$-operator can be used to construct the density
 $\boldsymbol{\langle} u,Q^g_k w \boldsymbol{\rangle}\in \Gamma (\ce[-n])$ for closed
 $k$-forms $ u,w \in\mathcal{E}^k$, and it is straightforward to
 show that this density is $Q$-like. Thus
 \begin{equation}\label{gp}
\int_M \boldsymbol{\langle} u,Q^g_k w \boldsymbol{\rangle} \bdg
 \end{equation}
is conformally invariant, as oberved in \cite{BrGo03,BrGo05}.

These observations are linked to two constructions of submanifold energies.

\subsubsection[The $Q_1$ operator]{The $\boldsymbol{Q_1}$ operator}

On manifolds of dimension four, the $Q_1$-operator maps closed
$1$-forms to $1$-forms of conformal weight $-2$. Let $(M,\cc)$ be a
conformal manifold of dimension four. For a fixed metric $g\in\cc$ the
operator $Q_1^g$ is given by
\begin{equation}\label{FirstQOperator}
    Q^g_1u_a=-\nabla_a\nabla^bu_b-4P_a{}^bu_b+2Ju_a,
\end{equation}
where $u_a$ is a closed $1$-form, and $\nabla$, $P$, and $J$ depend on
$g$. We have retrieved the formula above for $Q_1$ from the more
general formula for $Q_{\frac{n}{2}-1}$ on manifolds of dimension $n$,
which is computed in \cite{BrGo05}.

We verify here, explicitly, the key conformal property of $Q_1$.
\begin{prop}\label{ChangeQ1ua}
  Let $v_a\in\Gamma\left(\mathcal{E}_a\right)$ 
  be a closed $1$-form on a manifold of dimension four and $Q_1$ the $Q$-operator defined in Equation \eqref{FirstQOperator}. For a conformal transformation $g\mapsto  \hat{g}=e^{2\Upsilon}g$, where $\Upsilon\in C^\infty (M)$, $Q^g_1u_a$ transforms to
    \begin{equation*}
    Q^{\hat{g}}_1v_a=Q^g_1v_a+4\nabla^b\nabla_{[a}\left(\Upsilon v_{b]}\right).
    \end{equation*}
\end{prop}
\begin{proof}
    The conformal transformations below are with respect to $g\mapsto e^{2\Upsilon}g$. Recall that the conformal transformations of the Schouten tensor $P_{ab}$ and its trace $J$ are
    \begin{align*}
        P_{ab}\mapsto&P_{ab}-\nabla_a\Upsilon_b+\Upsilon_a\Upsilon_b-\frac{1}{2}g_{ab}\Upsilon^2\\
        J\mapsto&J-\nabla_c\Upsilon^c+\left(1-\frac{n}{2}\right)\Upsilon^2.
    \end{align*}
    By explicit computation is straightforward to find that
    \begin{equation*}
        \nabla^bv_b\mapsto\nabla^bv_b+(n-2)\Upsilon^bv_b,
    \end{equation*}
    so, setting $n=4$, the first term $-\nabla_a\nabla^bv_b$ of $Q_1v_a$ in Equation \eqref{FirstQOperator} transforms as
    \begin{align*}
        -\nabla_a\nabla^bv_b\mapsto&-\nabla_a\left(\nabla^bv_b+2\Upsilon^bv_b\right)+2\Upsilon_a\left(\nabla^bv_b+2\Upsilon^bv_b\right)\\
        \mapsto&-\nabla_a\nabla^bv_b-2\left(\nabla_a\Upsilon^b\right)v_b-2\Upsilon^b\nabla_av_b+2\Upsilon_a\nabla^bv_b+4\Upsilon_a\Upsilon^bv_b.
    \end{align*}
    Finally, using that $\Upsilon_a$ and $v_a$ are closed, we can
    write the conformal transformation of $Q_1v_a$ in the compact form
    given above. This is seen as follows:
    \begin{align*}
        Q^{\hat{g}}_1v_a=&Q^g_1v_a-2\left(\nabla_a\Upsilon^b\right)v_b-2\Upsilon^b\nabla_av_b+2\Upsilon_a\nabla^bv_b+4\Upsilon_a\Upsilon^bv_b\\
        &+4\left(\nabla_a\Upsilon^b\right)v_b-4\Upsilon_a\Upsilon^bv_b+2\Upsilon^2v_a\\
        &-2\left(\nabla_b\Upsilon^b\right)v_a-2\Upsilon^2v_a\\
        =&Q^g_1v_a+2\left(\nabla^b\Upsilon_a\right)v_b-2\Upsilon_b\nabla^bv_a+2\Upsilon_a\nabla^bv_b-2\left(\nabla^b\Upsilon_b\right)v_a\\
        =&Q^g_1v_a+2\nabla^b\left(\Upsilon_av_b\right)-2\nabla^b\left(\Upsilon_bv_a\right)\\
        =&Q^g_1v_a+4\nabla^b\left(\Upsilon_{[a}v_{b]}\right)\\
        =&Q^g_1v_a+4\nabla^b\nabla_{[a}\left(\Upsilon v_{b]}\right).
    \end{align*}
\end{proof}

Part of the importance of this transformation property is that it
enables the global pairing of closed 1-forms as in \eqref{gp}. For clairity 
we also verify this explicitly.
\begin{corr}\label{pair-c}
    On a manifold $M$ of dimension four let $Q_1$ be the $Q$-operator given in Equation \eqref{FirstQOperator}, and $u_a$, $v_a$ be two closed $1$-forms. Then the conformal transformation of $Q_1(u,v)$ is
    \begin{equation}\label{QlikeOp}
        Q^{\hat{g}}_1(u,v)=Q^g_1(u,v)+4\nabla^a\left(u^bv_{[a}\nabla_{b]}\Upsilon\right),
    \end{equation}
    and if $M$ is closed then 
    $$
\int_M Q^g_1(u,v) dg
$$
is conformally invariant. 
\end{corr}
\begin{proof}
    From Proposition \ref{ChangeQ1ua} above we know that
    \begin{equation*}
        u^aQ^{\hat{g}}_1v_a=u^aQ^g_1v_a+4u^a\nabla^b\nabla_{[a}\left(\Upsilon v_{b]}\right).
    \end{equation*}
    Using Leibniz on the second term we find that
    \begin{align*}
        u^a\nabla^b\nabla_{[a}\left(\Upsilon v_{b]}\right)=&\nabla^b\left(u^a\nabla_{[a}\left(\Upsilon v_{b]}\right)\right)-\left(\nabla^bu^a\right)\nabla_{[a}\left(\Upsilon v_{b]}\right)\\
        =&\nabla^a\left(u^bv_{[a}\nabla_{b]}\Upsilon\right)-\left(\nabla^{[b}u^{a]}\right)\nabla_a\left(\Upsilon v_b\right).
    \end{align*}
    Since $u_a$ is closed, the result follows. The final claim follows
    as $Q^g_1(u,v)$ takes values in the bundle $\mathcal{E}[-4]$.
\end{proof}

\subsubsection{A Global invariant for four-submanifolds}
Let $\Sigma^4\rightarrow M$ be a closed conformal submanifold of
dimension four, where $M^n$, $n\geq 5$, is a equipped with a (locally) flat
conformal structure $\cc$. There is a canonical conformally invariant
tractor-valued natural $1$-form on the submanifold, namely the tractor second
fundamental form $\mathbb{L}_{jA}{}^B$ defined in Equation
\eqref{TRACSECFF}. We use $\mathbb{L}_{jA}{}^B$ and the
$Q$-operator theory developed above to construct a $Q$-like density,
and thus a global conformal invariant on $\Sigma$.

Recall the tractor Codazzi-Mainardi Equation \eqref{ConformalCodazzi-Mainardi}:
\begin{equation*}
    \Omega_{ijAK}N^A_B=2\check{\nabla}_{[i}\mathbb{L}_{j]KB}.
\end{equation*}
Since $M$ is conformally flat, the ambient Weyl tensor and ambient Cotton tensor vanish. By Equation \eqref{AmbientCurvature2} this is equivalent to the ambient tractor curvature $\Omega_{abCD}$ vanishing, so we get that
\begin{equation*}
    \check{\nabla}_{[i}\mathbb{L}_{j]A}{}^B=0.
\end{equation*}
In other words, the tractor second fundamental form
$\mathbb{L}_{jA}{}^B$ is closed with respect to the normal-coupled
checked tractor connection $\check{\nabla}$. We shall say that it is
{\em $d^{\check{\nabla}}$-closed}. The $Q$-operator $Q_1$ for $\Sigma$
is then defined as in Equation \eqref{FirstQOperator}, except using
the intrinsic data of the submanifold: the intrinsic Schouten tensor
$p_{jk}$ and its trace $\jmath$. Moreover, we couple also the
connection operators in the formula for $Q_1$ to the checked tractor
connection, so that the resulting $Q_1:= Q_1^{\check{\nabla}}$ may act
on $d^{\check{\nabla}}$-closed tractor-valued $1$-forms. Since
$\check{\nabla}$ is an invariant connection, and no commutation of
derivatives was used the proof of Proposition \ref{ChangeQ1ua}, the resulting $Q_1$ operator
transforms as in Proposition
\ref{ChangeQ1ua}, but now acting on tractor-valued
$d^{\check{\nabla}}$-closed $1$-forms. This is a special case of the result \cite[{\bf Theorem 5.3}]{BrGo05}.

Using the above tools We define a canonical $Q$-like density for our
conformal submanifold: Let $\bar{g}$ be a fixed metric in the intrinsic
conformal class of $\Sigma$.
The $Q$-like density is
\begin{equation*}
    Q^{\bar{g}}_1\left(\mathbb{L},\mathbb{L}\right):=\mathbb{L}^{jA}{}_B\left(Q^{\bar{g}}_1\mathbb{L}\right){}_{jA}{}^B,
\end{equation*}
where tractor indices are raised and lowered using the ambient tractor
metric and tensor indices are raised and lowered (as usual) using the
induced conformal metric on $\Sigma$.

We are now ready to prove Theorem \ref{QoperatorEnergyThm}.
Note that, once again, no commuting of derivatives is used in the
Proof of Corollarly \ref{pair-c}. Thus it follows that
$Q^{\bar{g}}\left(\mathbb{L},\mathbb{L}\right)$ satisfies Equation
\eqref{QlikeOp}, meaning that under a conformal transformation
$\bar{g}\mapsto \hat{\bar{g}}=e^{2\Upsilon}\bar{g}$ we have
$$
Q^{\widehat{\bar{g}}}_1\left(\mathbb{L},\mathbb{L}\right)= Q^{\bar{g}}_1\left(\mathbb{L},\mathbb{L}\right) + 2 \nabla^i \left( \mathbb{L}^{jA}{}_B \mathbb{L}{}_{iA}{}^B \nabla_j \Upsilon -   \mathbb{L}^{jA}{}_B \mathbb{L}{}_{jA}{}^B \nabla_i \Upsilon\right) .
$$
Since also $Q^{\bar{g}}_1\left(\mathbb{L},\mathbb{L}\right) $ is a section of $\ce[-4]$ we have the following. 

\begin{prop}\label{Q-prop}
  On a an immersed closed 4-submanifold $\Sigma$  in a conformally flat manifold $(M^n,\cc)$, $n\geq 5$ the quantity
  \begin{equation*}
    \mathcal{E}:=\int_{\Sigma}Q^{\bar{g}}_1\left(\mathbb{L},\mathbb{L}\right)\bdg,
  \end{equation*}
  is an invariant.
\end{prop}

From the
definition of $Q_1$ an explicit formula for
$Q_1^{\bar{g}}\left(\mathbb{L},\mathbb{L}\right)$ with respect to some fixed
intrinsic metric $\bar{g}$ is
\begin{equation*}
    Q^{\bar{g}}_1\left(\mathbb{L},\mathbb{L}\right)=\mathbb{L}^{jA}{}_B\left(-\check{\nabla}_j\check{\nabla}^k\mathbb{L}_{kA}{}^B-4p_j{}^k\mathbb{L}_{kA}{}^B+2\jmath\mathbb{L}_{jA}{}^B\right).
\end{equation*}
Since $\Sigma$ is closed we can use integration by parts so that
\begin{equation*}
    \mathcal{E}=\int_{\Sigma}\left(\left(\check{\nabla}_j\mathbb{L}^{jA}{}_B\right)\check{\nabla}^k\mathbb{L}_{kA}{}^B-4p_j{}^k\mathbb{L}^{jA}{}_B\mathbb{L}_{kA}{}^B+2\jmath\mathbb{L}^{jA}{}_B\mathbb{L}_{jA}{}^B\right)\bdg,
\end{equation*}
and this gives Equation \eqref{SecondGWE}. We will prove in Section \ref{GJMS is Wt}  that $\mathcal{E}$ is of Willmore-type.

\subsection{GJMS Energy}\label{GJMS Energy section} The GJMS operators \cite{GJMS} are conformally invariant linear differential operators defined on densities. For conformal manifolds of even dimension $n$, the {\em $k^{th}$ GJMS operator}
$$P_{2k}:\Gamma\left(\mathcal{E}\left[k-\frac{n}{2}\right]\right)\longrightarrow\Gamma\left(\mathcal{E}\left[-k-\frac{n}{2}\right]\right), $$
where $k\in \mathbb{Z}_{\geq 1}$, and satisfies $2k\leq n$ if $n$ is even,   takes the form
\begin{equation*}
    P_{2k}=\Delta^k+\text{lower order derivatives},
\end{equation*}
where $\Delta$ is the Laplacian of the Levi-Civita connection in some
scale. The {\em critical GJMS operator} is the operator
$P_n:\Gamma\left(\mathcal{E}\right)\rightarrow\Gamma\left(\mathcal{E}[-n]\right)$.

For a conformal submanifold $\Sigma^m\rightarrow M^n$ with $m$ even
and $(M,\cc)$ conformally flat, write $P_m$ for the intrinsic critical
GJMS operator on $\Sigma$. We next, and henceforth, use the same notation (i.e. $P_m$)
to mean the coupling of $P_m$ to the ambient tractor connection. This means
that in an explicit formula for $P_m$ in terms of the Levi-Civita
connection of the submanifold (and its curvatures -- see
e.g. \cite{GoPe02,GJMS} for examples) we replace each occurrence of an
intrinsic Levi-Civita connection by its coupling to the ambient
tractor connection. 
\begin{prop}\label{inv-p}
    Let $\Sigma\rightarrow M$ be a conformal submanifold of dimension
    $m$, where $(M,\cc)$ is conformally flat. Then the intrinsic
    critical GJMS operator $P_m$, coupled to the ambient tractor
    connection, is conformally invariant on ambient tractor fields.
\end{prop}
\begin{proof}
  Let is prove it first for the case of $P_m$ acting on sections of the
  standard tractor bundle $\mathcal{E}^A$ (along $\Sigma$). 
    Let $U^A\in\mathcal{E}^A$ be an arbitrary ambient tractor field. We show first that $P_mU^A$ is conformally invariant on an arbitrary contractible open subset of $\Sigma$ in $M$.

    Since $M$ is conformally flat the curvature of the ambient tractor
    connection vanishes. This implies that on any contractible open
    subset $\mathcal{U}$ of a point $p\in \Sigma$ in $M$ there exists
    a parallel orthonormal basis $\left(E_1^A,\dots,E_{n+2}^A\right)$
    of the ambient tractor bundle restricted to $\mathcal{U}$. Write
    $U^A$ in this basis as
    \begin{equation*}
        U^A=f_1E_1^A+\cdots+f_{n+2}E{}_{n+2}^A
    \end{equation*}
    for smooth functions $f_1,\dots,f_{n+2}$ on $\mathcal{U}$.
   
    Given any formula for $P_m$, coupled to the ambient tractor
    connection, we find that $P_mU^A$ admits  the form
    \begin{equation*}
        P_mU^A=\left(P_mf_1\right)E_1^A+\cdots+\left(P_mf_{n+2}\right)E_{n+2}^A\in\left.\mathcal{E}^A[-m]\right|_{\mathcal{U}\cap \Sigma}.
    \end{equation*}
    Since $P_mf_i$ is conformally invariant for all $i$, it follows that $P_mU^A$ is conformally invariant on any contractible open subset $\mathcal{U}$ of $\Sigma$ in $M$. We conclude that $P_mU^A$ must be conformally invariant everywhere on $\Sigma$. 

This argument extends in the obvious way to $P_m$  acting on ambient tractors of higher rank.
\end{proof}

\begin{rem}\label{k-ob}
An alternative proof of Proposition \ref{inv-p} is simply to observe
that any calculation that verifies the conformal invariance of a
formula for $P_m$ when acting on functions is formally unchanged if instead we replace $P_m$ by its
coupling to a flat connection.

Note that for the approach here (and for Proposition \ref{Q-prop}), the fact that the
ambient tractor connection is flat is needed, as there is not a
formula for the dimension order GJMS operators that couples to a
general non-flat connection. It is straightforward to show that if
there were then we be able to use the tractor connection to construct
natural intrinsic conformal Laplacian power operators of order greater
than than the even dimension, in contradiction to
\cite{Graham-non,Gover-Hirachi-non}.
\end{rem}

It follows from Proposition \ref{inv-p} that, along an $m$-submanifold
$\Sigma$, the tractor $P_mN^B_A$ is well-defined and conformally
invariant with weight $-m$, where
$N^A_B\in\Gamma\left(\mathcal{N}^A_B\right)\subset\Gamma\left(\mathcal{E}^A_B\right)$
is the normal tractor projector, and thus $N^A_BP_mN^B_A$ is a
conformally invariant density of the same weight. This is the
appropriate weight to cancel with the volume form density $\bdg$. So we have the following result.
\begin{prop}\label{tE-prop}
  Let $\Sigma^m\rightarrow M$ be a closed submanifold of even
    dimension $m$ immersed in a conformally flat Riemannian manifold
    $(M,\cc)$ of dimension $n\geq m+1$. There is a conformally invariant
    $\Tilde{\mathcal{E}}$ on $\Sigma$  defined by
    \begin{equation}\label{GE}
        \Tilde{\mathcal{E}}:=\int_{\Sigma}N^A_BP_mN^B_A\bdg \, .
    \end{equation}
  \end{prop}
\noindent This invariant action \eqref{GE}  is what
we will refer to as the GJMS energy.

In fact the GJMS energy may also be viewed as arising from the
$Q_1$-opertor. The critical GJMS operator $P_m$ can be expressed in
terms of the $Q_1$-operator by $P_m=c_m\nabla^aQ_1\nabla_a$ where $c_m$ is some non-zero constant, see
\cite{BrGo03,BrGo05}.
Here as usual $\nabla^aQ_1\nabla_a$ means (up to a sign) the composition of $\delta$, $Q_1$ and $d$.  
 In the following we write $Q^{\nabla}_1$ and
$Q^{\check{\nabla}}_1$ for the intrinsic first $Q$-operator on
$\Sigma$ coupled to the ambient tractor connection $\nabla$ and the
checked tractor connection, respectively. Evidently we have the
following.
\begin{prop}\label{tE-prop}
  Let $\Sigma^m\rightarrow M$ be a closed submanifold of even
    dimension $m\geq4$ immersed in a conformally flat Riemannian manifold
    $(M,\cc)$ of dimension $n\geq 5$. The GJMS energy of $\Sigma$ can be expressed by
    \begin{equation}\label{GE}
        \Tilde{\mathcal{E}}:=\int_{\Sigma}(N^A_B\delta^\nabla Q^\nabla_1d^\nabla N^B_A)\, \bdg \, .
    \end{equation}
  \end{prop}
\noindent On conformally flat manifolds $(M,\cc)$ the ambient tractor
connection is flat and so $Q^{\nabla}_1 $ again has the transformation
property \eqref{Qop-t} when acting on tractor valued 1-forms that are
closed in the twisted sense. Clearly $d^\nabla N^B_A$ is an example of
the latter.

To complete
the proof of Theorem \ref{GJMSEnergyThm}, we will show in Section
\ref{GJMS is Wt} that the GJMS energy is of Willmore-type.

\subsection{Comparing our energies}\label{Comparing energies in dimension four}
We derive an equation relating the GJMS energy
$\Tilde{\mathcal{E}}$ and the $Q$-operator energy $\mathcal{E}$ for
submanifolds of dimension four.

To do this we first make some observations that apply in any even dimension $m$. These use the expression for  the
intrinsic critical GJMS operator $P_m$ coupled to the ambient tractor
connection in terms of the $Q^\nabla_1$-operator on submanifolds of even
dimension $m$, as discussed above. Let
$\Sigma^m\rightarrow M$ be a submanifold of dimension $m$
immersed in a conformally flat manifold $M$. We begin with the
following lemma (cf. \cite[Expression (3.29)]{CuGoS}).
\begin{lem}\label{DerivativeOfNis2L}
    The derivative of the normal tractor projector $N^A_B$ is
    \begin{equation*}
        \nabla_jN^A_B=-\mathbb{L}_j{}^A{}_B-\mathbb{L}_{jB}{}^A.
    \end{equation*}
\end{lem}
\begin{proof}
    We apply the Gauss-like formula in Equation \eqref{1tractorGaussEq1} relating the ambient tractor connection and the normal-coupled checked tractor connection to $N^A_B$. This gives
    \begin{align*}
        \nabla_jN^A_B=&\check{\nabla}_jN^A_B+\mathbb{L}_{jC}{}^AN^C_B-\mathbb{L}_j{}^A{}_CN^C_B+\mathbb{L}_j{}^C{}_BN^A_C-\mathbb{L}_{jB}{}^CN^A_C\\
        =&-\mathbb{L}_j{}^A{}_B-\mathbb{L}_{jB}{}^A,
    \end{align*}
    where $\check{\nabla}_jN^A_B=0$ follows from Lemma \ref{LastNeededReferencedLemma},    and $\mathbb{L}_{jC}{}^AN^C_B=0$ and $\mathbb{L}_j{}^A{}_CN^C_B=\mathbb{L}_j{}^A{}_B$ follow from Equation \eqref{TRACSECFF}.
\end{proof}

For a submanifold $\Sigma$ of dimension $m$, we couple $P_m$ to the
ambient tractor connection $\nabla$, so that (as discussed above) we have
$P_m=\nabla^jQ^{\nabla}_1\nabla_j$. The energy $\Tilde{\mathcal{E}}$
on $\Sigma$ is then
\begin{align*}
    \Tilde{\mathcal{E}}=&\int_{\Sigma}N^B_AP_mN^A_B\bdg\\
    =&\int_{\Sigma}N^B_A\nabla^jQ^{\nabla}_1\nabla_jN^A_B\bdg\\
    =&-\int_{\Sigma}\left(\nabla^jN^B_A\right)Q^{\nabla}_1\nabla_jN^A_B\bdg\\
    =&-\int_{\Sigma}\left(\mathbb{L}^{jB}{}_A+\mathbb{L}^j{}_A{}^B\right)Q^{\nabla}_1\left(\mathbb{L}_{jB}{}^A+\mathbb{L}_j{}^A{}_B\right)\bdg\\
    =&-2\int_{\Sigma}\left(\mathbb{L}^{jB}{}_AQ_1^{\nabla}\mathbb{L}_{jB}{}^A+\mathbb{L}^j{}_A{}^BQ_1^{\nabla}\mathbb{L}_{jB}{}^A\right)\bdg,
\end{align*}
where we have calculated in a scale,  assumed $\Sigma$ closed, and in the second to last line
we used Lemma \ref{DerivativeOfNis2L}.
In summary we have proved the following result.
\begin{prop}\label{tE-QL}
Let $\Sigma^m\rightarrow M$ be a closed submanifold of even dimension
$m$ immersed in a conformally flat Riemannian manifold $(M,\cc)$ of
dimension $n\geq 5$. The GJMS energy of $\Sigma$ can be expressed
\begin{equation}\label{tE-exp}
  \Tilde{\mathcal{E}}=-2\int_{\Sigma}\left(\mathbb{L}^{jB}{}_AQ_1^{\nabla}\mathbb{L}_{jB}{}^A+\mathbb{L}^j{}_A{}^BQ_1^{\nabla}\mathbb{L}_{jB}{}^A\right)\bdg .
  \end{equation}
  \end{prop}

\begin{rem}
The $Q_1$ operator has, by construction, order $m-2$. In view of
expression \eqref{ExplicitSFF} we see that the first term in the
integrand of expression \eqref{tE-exp} will in general involve $(m-2)$ derivatives acting on
$I\!I^{\circ}$. However notice that the second term in the integrand
cannot have  non-trvial contributions at this order.  This is because
the tractor second fundamental form on the left in the second term is
contracting its tangential and normal tractor indices with the normal
and tangential tractor indices of the other tractor second fundamental
form component, respectively. Again viewing expression
\eqref{ExplicitSFF} one sees that at least one derivative from $Q_1^\nabla$ must hit the
$N^A_a$ there in order to obtain a non-trivial contribution to the
integrand.
   
Or put another way this conclusion, for the second term, follows immediately if   use the tractor
Gauss-Weingarten formula \eqref{1tractorGaussEq1} to replace each $\nabla$ with
$\check{\nabla}$ plus lower order terms, as  the $\check{\nabla}$ covariant derivatives preserve the normal and tangential tractor bundles.

\end{rem}

We next find a formula relating $Q^{\nabla}_1$ and
$Q^{\check{\nabla}}_1$ in dimension $m=4$.  It is straightforward to use the Gauss-like
formula in Equation \eqref{1tractorGaussEq1} to express the tractors
$\nabla^k\mathbb{L}_{kB}{}^A$ and
$\nabla_j\nabla^k\mathbb{L}_{kB}{}^A$ in terms of the normal-coupled
checked tractor connection. The relevant equations are respectively
\begin{align*}
    \nabla^k\mathbb{L}_{kB}{}^A=&\check{\nabla}^k\mathbb{L}_{kB}{}^A+\mathbb{L}^{kC}{}_B\mathbb{L}_{kC}{}^A-\mathbb{L}^{kA}{}_C\mathbb{L}_{kB}{}^C
\end{align*}
and
\begin{align*}
    \nabla_j\nabla^k\mathbb{L}_{kB}{}^A=&\check{\nabla}_j\check{\nabla}^k\mathbb{L}_{kB}{}^A+\mathbb{L}_j{}^C{}_B\check{\nabla}^k\mathbb{L}_{kC}{}^A-\mathbb{L}_j{}^A{}_C\check{\nabla}^k\mathbb{L}_{kB}{}^C\\
    &+\check{\nabla}_j\left(\mathbb{L}^{kC}{}_B\mathbb{L}_{kC}{}^A\right)-\mathbb{L}_{jB}{}^D\mathbb{L}^{kC}{}_D\mathbb{L}_{kC}{}^A-\mathbb{L}_j{}^A{}_D\mathbb{L}^{kC}{}_B\mathbb{L}_{kC}{}^D\\
    &-\check{\nabla}_j\left(\mathbb{L}^{kA}{}_C\mathbb{L}_{kB}{}^C\right)-\mathbb{L}_{jD}{}^A\mathbb{L}^{kD}{}_C\mathbb{L}_{kB}{}^C-\mathbb{L}_j{}^D{}_B\mathbb{L}^{kA}{}_C\mathbb{L}_{kD}{}^C.
\end{align*}
It is then easy to see that
\begin{align*}
    \mathbb{L}^{jB}{}_A\nabla_j\nabla^k\mathbb{L}_{kB}{}^A=&\mathbb{L}^{jB}{}_A\check{\nabla}_j\check{\nabla}^k\mathbb{L}_{kB}{}^A\\
    &-\mathbb{L}^{jB}{}_A\mathbb{L}_{jB}{}^D\mathbb{L}^{kC}{}_D\mathbb{L}_{kC}{}^A-\mathbb{L}^{jB}{}_A\mathbb{L}_{jD}{}^A\mathbb{L}^{kD}{}_C\mathbb{L}_{kB}{}^C
\end{align*}
and
\begin{align*}
    \mathbb{L}^j{}_A{}^B\nabla_j\nabla^k\mathbb{L}_{kB}{}^A=&-\mathbb{L}^j{}_A{}^B\mathbb{L}_j{}^A{}_D\mathbb{L}^{kC}{}_B\mathbb{L}_{kC}{}^D-\mathbb{L}^j{}_A{}^B\mathbb{L}_j{}^D{}_B\mathbb{L}^{kA}{}_C\mathbb{L}_{kD}{}^C.
\end{align*}
Using Equation \eqref{TRACSECFF} we get that $\mathbb{L}^j{}_A{}^B\mathbb{L}_j{}^A{}_D\mathbb{L}^{kC}{}_B\mathbb{L}_{kC}{}^D=I\!I^{\circ j}{}_a{}^bI\!I^{\circ}{}_j{}^a{}_dI\!I^{\circ kc}{}_bI\!I^{\circ}{}_{kd}{}^c$ and $\mathbb{L}^j{}_A{}^B\mathbb{L}_j{}^D{}_B\mathbb{L}^{kA}{}_C\mathbb{L}_{kD}{}^C=I\!I^{\circ j}{}_a{}^bI\!I^{\circ}{}_j{}^d{}_bI\!I^{\circ ka}{}_cI\!I^{\circ}{}_{kd}{}^c$. Thus we find the following relations.
\begin{align*}
    \mathbb{L}^{jB}{}_AQ^{\nabla}_1\mathbb{L}_{jB}{}^A=&\mathbb{L}^{jB}{}_AQ^{\check{\nabla}}_1\mathbb{L}_{jB}{}^A\\
    &-I\!I^{\circ j}{}_a{}^bI\!I^{\circ}{}_j{}^a{}_dI\!I^{\circ kc}{}_bI\!I^{\circ}{}_{kd}{}^c-I\!I^{\circ j}{}_a{}^bI\!I^{\circ}{}_j{}^d{}_bI\!I^{\circ ka}{}_cI\!I^{\circ}{}_{kd}{}^c,
\end{align*}
and
\begin{equation*}
    \mathbb{L}^j{}_A{}^BQ_1^{\nabla}\mathbb{L}_{jB}{}^A=-I\!I^{\circ j}{}_a{}^bI\!I^{\circ}{}_j{}^a{}_dI\!I^{\circ kc}{}_bI\!I^{\circ}{}_{kd}{}^c-I\!I^{\circ j}{}_a{}^bI\!I^{\circ}{}_j{}^d{}_bI\!I^{\circ ka}{}_cI\!I^{\circ}{}_{kd}{}^c.
\end{equation*}
 The relation between $\Tilde{\mathcal{E}}$ and $\mathcal{E}$ is immediate. This is as follows.
\begin{lem}\label{Lemmm} For immersed 4-submanifolds $\Sigma$ we have:
 \begin{equation*}
    \Tilde{\mathcal{E}}=-2\mathcal{E}+4\int_{\Sigma}\left(I\!I^{\circ j}{}_a{}^bI\!I^{\circ}{}_j{}^a{}_dI\!I^{\circ kc}{}_bI\!I^{\circ}{}_{kd}{}^c+I\!I^{\circ j}{}_a{}^bI\!I^{\circ}{}_j{}^d{}_bI\!I^{\circ ka}{}_cI\!I^{\circ}{}_{kd}{}^c\right)\bdg.
 \end{equation*}
 \end{lem}
\noindent The display of the Lemma is Equation \eqref{SuperLastReferencedEquation}.

\begin{rem}\label{RemUmb}
    For the case when $\Sigma$ is umbilic, the integrands of the GJMS energy (in all dimensions) and of the $Q$-operator energy (in dimension four) are both zero. This is because for umbilic submanifolds the tractor second fundamental form vanishes completely, and how Equation \eqref{tE-exp} and Lemma \ref{Lemmm} express these integrands in terms of the tractor second fundamental form. Furthermore, such embeddings of $\Sigma$ are critical for these energies. Equation \eqref{tE-exp} above shows that $\Tilde{\mathcal{E}}$ is quadratic in the tractor second fundamental form, so it follows that any variation of embedding of this energy is necessarily zero. Umbilic submanifolds are critical for the $Q$-operator energy by the same reasoning, or by Lemma \ref{Lemmm}.
\end{rem}

\subsection[Comparing the $Q$-operator and Graham-Reichert energies]{Comparing the $\boldsymbol{Q}$-operator and Graham-Reichert energies}\label{GRcompare}

In this section we will compute the difference between the $Q$-operator energy and the Graham-Reichert energy. To do this we use the Fialkow tensor defined above, and the Chern-Gauss-Bonnet formula for manifolds of dimension four.

We wrote in the introduction that the Graham-Reichert energy is given in our notation by
\begin{equation*}
    \begin{split}
        8\mathcal{E}_{GR}=\int_{\Sigma}\biggl(&(D_jH_b-P_{ja}N^a_b)(D^jH^b-P^{jc}N^b_c)-|\check{P}|^2+\check{J}\\
    &-W^k{}_{akb}H^aH^b-2C^k{}_{ka}H^a-\frac{1}{n-4}B^k{}_k\biggr)\bdg.
    \end{split}
\end{equation*}
In Equation \ref{ContractConfRicc} we will show that
\begin{equation*}
    D^kI\!I^{\circ}{}_{jk}{}^b=(m-1)\left(D_jH^b-\Pi_j^cN^b_dP_c{}^d\right),
\end{equation*}
when the ambient manifold is conformally flat, so in the setting of our energies $\Tilde{\mathcal{E}}$ and $\mathcal{E}$ defined above, the Graham-Reichert energy is given by
\begin{equation*}
    8\mathcal{E}_{GR}=\int\left(\frac{1}{9}\left(D^kI\!I^{\circ}{}_{jka}\right)D^lI\!I^{\circ j}{}_l{}^a-|\check{P}|^2+\check{J}^2\right)\bdg.
\end{equation*}
Immediately we see that the Graham-Reichert and $Q$-operator energies are related by
\begin{equation*}
    32\mathcal{E}_{GR}-\mathcal{E}=\int_{\Sigma}\left(-4|\check{P}|^2+4\check{J}^2+4p_j{}^k\mathbb{L}^{jA}{}_B\mathbb{L}_{kA}{}^B-2\jmath\mathbb{L}^{jA}{}_B\mathbb{L}_{jA}{}^B\right)\bdg.
\end{equation*}
To see that this difference is conformally invariant consider the following. The Fialkow tensor $\mathcal{F}_{jk}:=\check{P}_{jk}-p_{jk}$ is defined as the difference between the two submanifold Schouten tensors $\check{P}_{jk}$ and $p_{jk}$, so we see that
\begin{equation*}
    |\check{P}|^2=|\mathcal{F}|^2+2\mathcal{F}_{ij}p^{ij}+|p|^2,
\end{equation*}
and
\begin{equation*}
    \check{J}^2=\mathrm{f}^2+2\mathrm{f}\jmath+\jmath^2,
\end{equation*}
where $\mathrm{f}:=\mathcal{F}_{kl}\bg^{kl}$. It can be shown using the tractor Gauss Equation \ref{CGE}, or otherwise, that the Fialkow tensor, for $m\geq3$, has the formula
\begin{equation*}
    \mathcal{F}_{ij}=\frac{1}{m-2}\left(W_{iajb}N^{ab}+\frac{W_{abcd}N^{ac}N^{bd}}{2(m-1)}\bg_{ij}+I\!I^{\circ}{}_i{}^{kc}I\!I^{\circ}{}_{jkc}-\frac{I\!I^{\circ klc}I\!I^{\circ}{}_{klc}}{2(m-1)}\bg_{ij}\right),
\end{equation*}
see \cite[{\bf Section 3.4}]{CuGoS} for computations. In the setting above where $W=0$ and $m=4$, it is not hard to show that
\begin{equation*}
    -4|\check{P}|^2+4\check{J}^2=-4|\mathcal{F}|^2+4\mathrm{f}^2-4|p|^2+4\jmath^2-4I\!I^{\circ}{}_i{}^{kc}I\!I^{\circ}{}_{jkc}p^{ij}+2I\!I^{\circ klc}I\!I^{\circ}{}_{klc}\jmath.
\end{equation*}
We can now rewrite the difference of the two energies above as
\begin{align*}
    32\mathcal{E}_{GR}-\mathcal{E}=&\int_{\Sigma}\left(-4|\mathcal{F}|^2+4\mathrm{f}^2+\frac{1}{2}e(\Omega)-\frac{1}{2}w_{ijkl}w^{ijkl}\right)\bdg\\
    =&16\pi^2\chi(\Sigma)+\int_{\Sigma}\left(-4|\mathcal{F}|^2+4\mathrm{f}^2-\frac{1}{2}w_{ijkl}w^{ijkl}\right)\bdg,
\end{align*}
where $e(\Omega):=-8\left(|p|^2-\jmath^2\right)+w_{ijkl}w^{ijkl}$ is the Pfaffian of the submanifold Riemannian curvature in some choice of scale, $\chi(\Sigma)$ is the Euler characteristic of $\Sigma$, and we have used the Chern-Gauss-Bonnet formula to write
\begin{equation*}
    \int_{\Sigma}e(\Omega)\bdg=32\pi^2\chi(\Sigma).
\end{equation*}
See \cite[{\bf Section 6.4}]{GR} for an application of the Chern-Gauss-Bonnet formula in this context.

\subsection[GJMS and $Q$-operator energies are of Willmore-type]{GJMS and $\boldsymbol{Q}$-operator energies are of Willmore-type}\label{GJMS is Wt}

Here we will show that the GJMS energy is of Willmore-type. Recall from the above subsection that the GJMS energy $\Tilde{\mathcal{E}}$ can be expressed in terms of the first $Q$-operator $Q_1^{\nabla}$ coupled to the ambient tractor connection $\nabla$ by
\begin{equation*}
    \Tilde{\mathcal{E}}=-2\int_{\Sigma}\left(\mathbb{L}^{jB}{}_AQ_1^{\nabla}\mathbb{L}_{jB}{}^A+\mathbb{L}^j{}_A{}^BQ_1^{\nabla}\mathbb{L}_{jB}{}^A\right)\bdg.
\end{equation*}
In the subsequent discussion we will take $\Sigma $ to be closed.
We remarked above that the second term in the integrand does not
contribute to the highest order term of the energy via applications of
the tractor Gauss and Weingarten formulae. 
By this same reasoning, the
density
\begin{equation*}
    \mathbb{L}^{jB}{}_AQ_1^{\check{\nabla}}\mathbb{L}_{jB}{}^A
\end{equation*}
is a summand of the first term, and contains the highest order term of
the integrand. By Equation \eqref{FormOfQ} the $Q_1$-operator
$Q_1^{\check{\nabla}}$, coupled to the checked tractor connection, has
the form
\begin{equation*}
    Q_1^{\check{\nabla}}u_j= \alpha \check{\nabla}_j\check{\Delta}^{m/2-2}\check{\nabla}^ku_k+lots,
\end{equation*}
where $u_j$ is a tractor valued $1$-form and $\alpha$ is a nonzero constant. Thus the highest order term must, up to lower order terms, be
\begin{equation}\label{FindThis}
    \left(\check{\nabla}^j\mathbb{L}_j{}^A{}_B\right)\check{\Delta}^{m/2-2}\check{\nabla}^k\mathbb{L}_{kA}{}^B,
\end{equation}
where we have applied integration by parts to move $\check{\nabla}$
onto the left-most $\mathbb{L}$. Recall the splitting of
$\mathbb{L}_{jK}{}^A$ from Equation \eqref{ExplicitSFF},
\begin{equation*}
    \mathbb{L}_{jK}{}^A=\left(\begin{array}{c}
        0\\
        I\!I^{\circ}{}_{jk}{}^a\\
        -D_jH^a+\Pi^b_jN^a_cP_b{}^c
    \end{array}\right)N^A_a,
\end{equation*}
and recall that $\check{\nabla}_jN^A_a=0$ by Lemma \ref{LastNeededReferencedLemma}. An application of the checked tractor connection $\check{\nabla}$ to $\mathbb{L}_{jK}{}^A$ via Equation \eqref{CheckedProjectors} is shown below.
\begin{equation*}
    \check{\nabla}_i\mathbb{L}_{jK}{}^A=\left(\begin{array}{c}
        -I\!I^{\circ}{}_{ij}{}^a\\
        D_iI\!I^{\circ}{}_{jk}{}^a-g_{ik}\left(D_jH^a-\Pi^c_jN^a_dP_c{}^d\right)\\
        -D_iD_jH^a+D_i\left(\Pi^b_jN^a_cP_b{}^c\right)-\check{P}_i{}^kI\!I^{\circ}{}_{jk}{}^a
    \end{array}\right)N^A_a.
\end{equation*}
Since $\mathbb{L}$ is $d^{\check{\nabla}}$-closed, we have that $Z^K_kN^a_A\check{\nabla}_{[i}\mathbb{L}_{j]K}{}^A=0$. This is equivalent to the following equation.
\begin{equation*}
    D_{[i}I\!I^{\circ}{}_{j]k}{}^a-g_{k[i}D_{j]}H^a+g_{k[i}\Pi^c_{j]}N^a_dP_c{}^d=0.
\end{equation*}
Tracing $i$ and $k$ gives
\begin{equation}\label{ContractConfRicc}
    D^kI\!I^{\circ}{}_{jk}{}^b=(m-1)\left(D_jH^b-\Pi_j^cN^b_dP_c{}^d\right).
\end{equation}
We see that the $\check{\nabla}-$divergence of the the tractor second fundamental form has the explicit formula
\begin{align*}
    \check{\nabla}^j\mathbb{L}_{jK}{}^A=&\left(\begin{array}{c}
        0\\
        D^jI\!I^{\circ}{}_{jk}{}^a-\left(D_kH^a-\Pi^c_kN^a_dP_c{}^d\right)\\
        -D_jD^jH^a+D^j\left(\Pi^b_jN^a_cP_b{}^c\right)-\check{P}^{jk}I\!I^{\circ}{}_{jk}{}^a
    \end{array}\right)N^A_a.
\end{align*}
Substituting Equation \eqref{ContractConfRicc} into the above then gives
\begin{equation*}
    \check{\nabla}^j\mathbb{L}_{jK}{}^A=\left(\begin{array}{c}
        0\\
        (m-2)\left(D_kH^a-\Pi^c_kN^a_dP_c{}^d\right)\\
        -D_jD^jH^a+D^j\left(\Pi^b_jN^a_cP_b{}^c\right)-\check{P}^{jk}I\!I^{\circ}{}_{jk}{}^a
    \end{array}\right)N^A_a.
\end{equation*}
To simplify computations we will omit all terms except those with the highest possible order of derivatives of the immersion in each slot. For example, we will write the above as
\begin{equation}\label{Edgecase}
    \check{\nabla}^j\mathbb{L}_{jK}{}^A\stackrel{\cdot}{=}\left(\begin{array}{c}
        0\\
        (m-2)D_kH^a\\
        -D_jD^jH^a
    \end{array}\right)N^A_a.
\end{equation}
When we apply the tractor connection to Equation \eqref{Edgecase} we get
\begin{equation*}
    \check{\nabla}_i\check{\nabla}^j\mathbb{L}_{jK}{}^A\stackrel{\cdot}{=}\left(\begin{array}{c}
        -(m-2)D_iH^a\\
        (m-2)D_iD_kH^a-\bg_{ik}D_jD^jH^a\\
        -D_iD_jD^jH^a
    \end{array}\right)N^A_a,
\end{equation*}
where the terms contributing Schouten tensor components are quadratic and therefore of lower order. One more application of the tractor connection gives us
\begin{equation*}
    \check{\Delta}\check{\nabla}^j\mathbb{L}_{jK}{}^A \stackrel{\cdot}{=}\left(\begin{array}{c}
        -(m-4)D_jD^jH^a\\
        (m-4)D_kD_jD^jH^a\\
        -\left(D_jD^j\right)^2H^a
    \end{array}\right)N^A_a,
\end{equation*}
where the commutation of derivatives contributes curvature terms which
are of lower order, and are therefore omitted.  We make the following
observation.
\begin{lem}\label{InductionLemma}
    Let $r\geq1$ be an integer. Then
    \begin{equation*}
        \check{\Delta}^{r-1}\check{\nabla}^j\mathbb{L}_{jK}{}^A \stackrel{\cdot}{=} \left(\begin{array}{c}
            -(r-1)(m-2r)\left(D_jD^j\right)^{r-1}H^a\\
            (m-2r)D_k\left(D_jD^j\right)^{r-1}H^a\\
            -\left(D_jD^j\right)^rH^a
        \end{array}\right)N^A_a
    \end{equation*}
    for all $r$, where terms of lower order in each slot (with respect to the splitting) are omitted.
\end{lem}
\begin{proof}
    We have already shown above that this holds true for $r=1$ and $r=2$. We prove the lemma by induction. Suppose for $r=l-1$ the above is true for some integer $l>2$. Then by assumption we have
    \begin{equation*}
        \check{\Delta}^{l-2}\check{\nabla}^j\mathbb{L}_{jK}{}^A=\left(\begin{array}{c}
            -(l-2)(m-2l+2)\left(D_jD^j\right)^{l-2}H^a\\
            (m-2l+2)D_k\left(D_jD^j\right)^{l-2}H^a\\
            -\left(D_jD^j\right)^{l-1}H^a
        \end{array}\right)N^A_a.
    \end{equation*}
    The first application of the tractor connection gives
    \begin{equation*}
        \check{\nabla}_i\check{\Delta}^{l-2}\check{\nabla}^j\mathbb{L}_{jK}{}^A=\left(\begin{array}{c}
            -(l-1)(m-2l+2)D_i\left(D_jD^j\right)^{l-2}H^a\\
            (m-2l+2)D_iD_k\left(D_jD^j\right)^{l-2}H^a-\bg_{ik}\left(D_jD^j\right)^{l-1}H^a\\
            -D_i\left(D_jD^j\right)^{l-1}H^a
        \end{array}\right)N^A_a,
    \end{equation*}
    and the second application gives
    \begin{align*}
      \check{\Delta}^{l-1}\check{\nabla}^j\mathbb{L}_{jK}{}^A   \hspace*{-1mm} =
      &\left(  \hspace*{-1mm} \begin{array}{c}
         \hspace*{-1mm}   -(l-1)(m-2l+2)\left(D_jD^j\right)^{l-1}H^a+2(l-1)\left(D_jD^j\right)^{l-1}H^a\\
            (m-2l+2)D_k\left(D_jD^j\right)^{l-1}H^a-2D_k\left(D_jD^j\right)^{l-1}H^a\\
            -\left(D_jD^j\right)^lH^a
        \end{array} \hspace*{-1mm}  \right)N^A_a\\
        =&\left(\begin{array}{c}
            -(l-1)(m-2l)\left(D_jD^j\right)^{l-1}H^a\\
            (m-2l)D_k\left(D_jD^j\right)^{l-1}H^a\\
            -\left(D_jD^j\right)^lH^a
        \end{array}\right)N^A_a.
    \end{align*}
        This shows that the claim is true for $r=l$, and therefore by induction is true for all $r$.
\end{proof}

\begin{proof}[Proof of Theorems \ref{GJMSEnergyThm} and \ref{QoperatorEnergyThm}]
In Section \ref{Comparing energies in dimension four} we showed that
the $Q$-operator energy and the GJMS energy in dimension four differ
only by low-order terms and therefore share the same highest-order
term, up to multiplication by a constant. It is thus sufficient to
prove that the GJMS energy in all even dimensions $m\geq4$ is of
Willmore-type, which by our comparison will show that the $Q$-operator
energy is of Willmore-type. The remaining properties in Theorems
\ref{GJMSEnergyThm} and \ref{QoperatorEnergyThm} such as conformal
invariance and an explicit formula have been proved already in the
respective Sections \ref{GJMS Energy section} and
\ref{QOperatorEnergySection}.

Following our computations at the beginning of this subsection, we must find the highest order term in \eqref{FindThis} and show that this term is $$H_a\left(D_jD^j\right)^{m/2-1}H^a,$$
up to multiplication by a non-zero constant and integration by parts. By Lemma \ref{InductionLemma} we have 
\begin{equation*}
    \check{\Delta}^{m/2-2}\check{\nabla}^j\mathbb{L}_{jK}{}^A\stackrel{\cdot}{=}\left(\begin{array}{c}
        -\left(m-4\right)\left(D_jD^j\right)^{m/2-2}H^a\\
        2D_k\left(D_jD^j\right)^{m/2-2}H^a\\
        -\left(D_jD^j\right)^{m/2-1}H^a
    \end{array}\right)N^A_a.
\end{equation*}
Thus \eqref{FindThis} is, up to lower-order terms, given by
\begin{align*}
    \left(\check{\nabla}^i\mathbb{L}_i{}^K{}_A\right)\check{\Delta}^{m/2-2}\check{\nabla}^j\mathbb{L}_{jK}{}^A=&\left(\begin{array}{c}
        0\\
        (m-2)D^kH_a\\
        -D_jD^jH_a
    \end{array}\right)\cdot\left(\begin{array}{c}
        -\left(m-4\right)\left(D_jD^j\right)^{m/2-2}H^a\\
        2D_k\left(D_jD^j\right)^{m/2-2}H^a\\
        -\left(D_jD^j\right)^{m/2-1}H^a
    \end{array}\right)\\
    =&2(m-2)\left(D^kH_a\right)D_k\left(D_jD^j\right)^{m/2-2}H^a\\
    &+(m-4)\left(D_kD^kH_a\right)\left(D_jD^j\right)^{m/2-2}H^a,
\end{align*}
where we have used the tractor inner product \ref{tmf}. Finally, integration by parts shows that the highest-order term of the GJMS energy is
\begin{equation*}
    mH_a\left(D_jD^j\right)^{m/2-1}H^a,
\end{equation*}
up to multiplication of some non-zero constant.
\end{proof}

\Addresses


\clearpage

\begin{thebibliography}{99}

\bibitem{ACF}
\newblock L. Andersson, P.T. Chruściel, and H. Friedrich,
\newblock\textit{On the regularity of solutions to the Yamabe equation and the existence of smooth hyperboloidal initial data for Einstein's field equations}.
\newblock Comm. Math. Phys. \textbf{149}, (1992) 587–612,
\newblock \url{https://doi.org/10.1007/BF02096944}.

\bibitem{Aubry-Guillarmou}
\newblock E. Aubry and C. Guillarmou,
\newblock\textit{Conformal harmonic forms, Branson–Gover operators and Dirichlet problem at infinity}.
\newblock Eur. Math. Soc. \textbf{13}(4), (2011) 911–957,
\newblock \url{https://doi.org/10.4171/JEMS/271}.

\bibitem{Blaschke} W. Blaschke,
\newblock \textit{Vorlesungen über Differentialgeometrie und geometrische Grundlagen von Einsteins Relativitätstheorie III},
\newblock Springer Berlin, Heidelberg, (1929),
\newblock \url{https://doi.org/10.1007/978-3-642-50823-3}.

\bibitem{BlGoWa} S. Blitz, A.R. Gover, and A. Waldron,
\newblock \textit{Generalized Willmore energies, Q-curvatures, extrinsic Paneitz operators, and extrinsic Laplacian powers},
\newblock Comm. Cont. Math. \textbf{26}(5), (2024) 2350014,
\newblock \url{https://doi.org/10.1142/S0219199723500141}.

\bibitem{Blitz-Silhan}
\newblock S. Blitz and J. Šilhan,
\newblock\textit{Holography of Higher Codimension Submanifolds: Riemannian and Conformal}.
\newblock Preprint: \url{math.DG/2405.07692}, \url{https://arxiv.org}.

\bibitem{tom-sharp} 
  \newblock T. Branson,
  \newblock \textit{Sharp inequalities, the functional determinant, and the complementary series}.
  \newblock Trans. Am.
Math. Soc. \textbf{347}, (1995) 3671-3742,
\newblock\url{https://doi.org/10.2307/2155203}. 

\bibitem{BrGo03} T. Branson and A.R. Gover,
\newblock \textit{Conformally invariant operators, differential forms, cohomology and a generalisation of Q-curvature}.
\newblock Comm. Partial Differential Equations \textbf{30}, (2005) 1611-1669,
\newblock \url{https://doi.org/10.1080/03605300500299943}.

\bibitem{BrGo05} T. Branson and A.R. Gover,
\newblock\textit{Pontrjagin forms and invariant objects related to the Q-curvature}.
\newblock Comm. Cont. Math. \textbf{9}(3), (2007) 335-358,
\newblock \url{https://doi.org/10.1142/S0219199707002460}.

\bibitem{BrOsted} T.P. Branson and B. Ørsted,
\newblock\textit{Explicit Functional Determinants in Four Dimensions}.
\newblock Proc. Am. Math. Soc. \textbf{113}(3), (1991) 669-682,
\newblock \url{https://doi.org/10.2307/2048601}.

\bibitem{CuThesis} S. Curry,
\newblock\textit{Submanifolds in Conformal and CR Manifolds and Applications}.
\newblock Doctoral thesis, University of Auckland, (2016),
\newblock Eprint: \url{http://hdl.handle.net/2292/29166}.

\bibitem{CuGo13} S. Curry and A.R. Gover,
\newblock\textit{An introduction to conformal geometry and tractor calculus, with a view to applications in general relativity}.
\newblock In: T. Daudé, D. Häfner, J.P. Nicolas (eds.):
\newblock\textit{Asymptotic Analysis in General Relativity}.
\newblock London Math. Soc. Lecture Note Series. Cambridge University Press, (2018) 86-170,
\newblock \url{https://doi.org/10.1017/9781108186612.003}.

\bibitem{CuGoS} S.N. Curry, A.R. Gover, and D. Snell,
\newblock\textit{Conformal submanifolds, distinguished submanifolds, and integrability}.
\newblock Preprint: \url{math.DG/2309.09361}, \url{https://arxiv.org}.

\bibitem{G-aE} A.R. Gover,
\newblock\textit{Almost Einstein and Poincare-Einstein manifolds in
  Riemannian signature}.
  \newblock J. Geom. Phys. \textbf{60}(2), (2010) 182-204,
  \newblock \url{https://doi.org/10.1016/j.geomphys.2009.09.016}.

\bibitem{Go04} A.R. Gover,
\newblock\textit{Conformal de Rham Hodge theory and operators generalising the Q-curvature}.
\newblock In: Slovák, Jan and Čadek, Martin (eds.):
\newblock\textit{Proceedings of the 24th Winter School "Geometry and Physics"}.
\newblock Circolo Matematico di Palermo, Palermo, (2005) 109-137,
\newblock \url{http://hdl.handle.net/10338.dmlcz/701745}.

\bibitem{Gover-Hirachi-non}  A.R.\ Gover and K.\ Hirachi,
  \newblock\textit{Conformally invariant powers of the Laplacian—a complete nonexistence theorem},
\newblock J. Amer. Math. Soc. \textbf{17}, (2004) 389-405,
\newblock\url{https://doi.org/10.1090/S0894-0347-04-00450-3}.

\bibitem{GoPe02} A. Gover and L. Peterson,
\newblock\textit{Conformally Invariant Powers of the Laplacian, Q-Curvature, and Tractor Calculus}.
\newblock Commun. Math. Phys. \textbf{235}, (2003) 339-378,
\newblock \url{https://doi.org/10.1007/s00220-002-0790-4}.

\bibitem{GoWa16-2} A.R. Gover and A. Waldron,
\newblock\textit{A calculus for conformal hypersurfaces and new higher Willmore energy functionals}.
\newblock Adv. Geom. \textbf{20}(1), (2020) 29-60,
\newblock \url{https://doi.org/10.1515/advgeom-2019-0016}.

\bibitem{GW-LN}
A.R. Gover and A. Waldron.
\newblock \textit{Conformal
  hypersurface geometry via a boundary Loewner-Nirenberg-Yamabe problem}.
\newblock Comm. Analysis
    Geom. \textbf{29}(4), (2021) 779-836,
    \newblock \url{https://dx.doi.org/10.4310/CAG.2021.v29.n4.a2}.

\bibitem{GW-ann} A.R.\ Gover and A.\ Waldron \newblock\textit{ Generalising
  the Willmore equation: submanifold conformal invariants from a
  boundary Yamabe problem}.
\newblock Preprint: \href{https://arxiv.org/abs/1407.6742}{hep-th/1407.6742}, \url{https://arxiv.org}.

\bibitem{GoWa16-1} A.R. Gover and A. Waldron,
\newblock\textit{Renormalized volume}.
\newblock Comm. Math. Phys. \textbf{354}, (2017) 1205–1244,
\newblock \url{https://doi.org/10.1007/s00220-017-2920-z}.

\bibitem{GoWa16-3} A.R. Gover and A. Waldron,
\newblock\textit{Renormalized volumes with boundary}.
\newblock Comm. Cont. Math. \textbf{21}(2), (2019) 1850030,
\newblock \url{https://doi.org/10.1142/S021919971850030X}.

\bibitem{GoWa19} A.R. Gover and A.K. Waldron,
\newblock\textit{Singular Yamabe and Obata problems}.
\newblock In: O. Dearricott et al. (eds.):
\newblock\textit{Differential Geometry in the Large}.
\newblock Cambridge, Cambridge University Press (London Math. Soc. Lecture Note Series), (2020) 193–214,
\newblock \url{https://doi.org/10.1017/9781108884136.011}.

\bibitem{GW-Barcelona} 
A.R. Gover and A. Waldron,
\newblock\textit{Submanifold
  conformal invariants and a boundary Yamabe problem}.
\newblock In: M. González, P. Yang, N. Gambino, J. Kock (eds):
\newblock\textit{Extended Abstracts Fall 2013}.
\newblock Trends in Mathematics. Birkhäuser, Cham, (2013) 21-26,
\newblock \url{ https://doi.org/10.1007/978-3-319-21284-5_4}.

\bibitem{Graham-non} C.R. Graham,
  \newblock\textit{Conformally invariant powers of the Laplacian. II.
    Nonexistence}. 
\newblock J. London Math. Soc. \textbf{46}(2), (1992) 566-576,
\newblock\url{ https://doi.org/10.1112/jlms/s2-46.3.566}.



\bibitem{Gr-Vol} C.R. Graham,
\newblock\textit{Volume renormalization for singular Yamabe metrics}.
\newblock Proc. Am. Math. Soc. \textbf{145}(4), (2017) 1781-1792,
\newblock \url{http://dx.doi.org/10.1090/proc/13530}.
  
\bibitem{GJMS} C.R. Graham, R. Jenne, L.J. Mason, and G.A.J. Sparling,
\newblock\textit{Conformally invariant powers of the Laplacian, I:Existence}.
\newblock J. London Math. Soc. (2) \textbf{46} (1992) 557-565,
\newblock \url{https://doi.org/10.1112/jlms/s2-46.3.557}.

    \bibitem{GR} C.R. Graham and N. Reichert,
\newblock\textit{Higher-dimensional Willmore energies via minimal submanifold asymptotics}.
\newblock Asian J.Math. \textbf{24}(4), (2020) 571-610,
\newblock \url{https://dx.doi.org/10.4310/AJM.2020.v24.n4.a3}.


    \bibitem{GW} C.R. Graham and E. Witten,
\newblock\textit{Conformal anomaly of submanifold observables in
      AdS/CFT correspondence}.
 \newblock   Nuclear Phys. B \textbf{546}(1), (1999) 52-64,
 \newblock \url{https://doi.org/10.1016/S0550-3213(99)00055-3}.

\bibitem{Gu05} J. Guven,
\newblock\textit{Conformally invariant bending energy for hypersurfaces}.
\newblock J. Phys. A: Math. Gen. \textbf{38}(37), (2005) 7943-7955,
\newblock \url{https://doi.org/10.1088/0305-4470/38/37/002}.

\bibitem{JuhlOrst}
\newblock A. Juhl and B. Ørsted,
\newblock\textit{On singular Yamabe obstructions}.
\newblock J. Geom. Analysis \textbf{32}, (2022) 146,
\newblock\url{https://doi.org/10.1007/s12220-022-00867-6}.

\bibitem{McKoewnetal}
\newblock S.R.C. Kushtagi and S.E. McKeown,
\newblock\textit{Volume renormalization of higher-codimension singular Yamabe spaces}.
\newblock Preprint: \url{math.DG/2405.07692}, \url{https://arxiv.org}.

\bibitem{MaNe} F.C. Marques and A. Neves,
\newblock\textit{Min-max theory and the Willmore conjecture}.
\newblock Annals of Math. \textbf{179}, (2014) 683-782,
\newblock \url{https://doi.org/10.4007/annals.2014.179.2.6}.

\bibitem{Mar23}
D. Martino,
\newblock\textit{A duality theorem for a four dimensional Willmore energy}. 2023;
\newblock Eprint: \href{https://arxiv.org/abs/2308.11433}{math.DG/2308.11433}, \url{https://arxiv.org}.

\bibitem{V} Y. Vyatkin,
\newblock\textit{Manufacturing conformal invariants of hypersurfaces}.
\newblock Doctoral thesis, University of Auckland, (2013),
\newblock Eprint: \url{http://hdl.handle.net/2292/22320}.

\bibitem{WillmoreArticle} T.J. Willmore,
\newblock\textit{Note on embedded surfaces}.
\newblock Ann. of the "Alexandru Ioan Cuza" Uni. of Iaşi \textbf{Tomul XI}, (1965) 493-496,
\newblock\url{https://www.math.uaic.ro/~annalsmath/index.php/publication-links/remarkable-papers/}.

\bibitem{WillmoreBook} T.J. Willmore,
\newblock\textit{Riemannian Geometry}.
\newblock Oxford, (1993); online edition, Oxford Academic, (2023),
\newblock \url{https://doi.org/10.1093/oso/9780198532538.001.0001}.

\bibitem{YZ} Y. Zhang,
  \newblock\textit{Graham-Witten's conformal invariant for closed
  four-dimensional submanifolds}.
\newblock J. Math. Study \textbf{54}(2), (2021) 200-226,
\newblock\url{https://doi.org/10.4208/jms.v54n2.21.06}.

\end{thebibliography}
\end{document}